\documentclass[onefignum,onetabnum]{siamart190516}

\usepackage{lipsum}
\usepackage{amsfonts}
\usepackage{graphicx}
\usepackage{epstopdf}
\usepackage{algorithmic}
\ifpdf
  \DeclareGraphicsExtensions{.eps,.pdf,.png,.jpg}
\else
  \DeclareGraphicsExtensions{.eps}
\fi

\usepackage{amssymb}
\usepackage[caption=false]{subfig}

\newcommand{\R}{\mathbb{R}}
\newcommand{\N}{\mathbb{N}}
\DeclareMathOperator*{\sign}{sign}
\DeclareMathOperator*{\argmin}{arg\,min}
\DeclareMathOperator*{\rk}{rk}
\DeclareMathOperator*{\spn}{span}
\DeclareMathOperator*{\pr}{pr}
\DeclareMathOperator*{\conv}{conv}
\DeclareMathOperator*{\aff}{aff}
\DeclareMathOperator*{\affdim}{affdim}
\DeclareMathOperator*{\interior}{int}
\DeclareMathOperator*{\cl}{cl}
\DeclareMathOperator*{\ri}{ri}
\DeclareMathOperator*{\rb}{rb}

\newsiamremark{remark}{Remark}
\newsiamremark{hypothesis}{Hypothesis}
\crefname{hypothesis}{Hypothesis}{Hypotheses}
\newsiamthm{claim}{Claim}

\newcommand{\newsiamass}[2]{
  \theoremstyle{plain}
  \theoremheaderfont{\normalfont\sc}
  \theorembodyfont{\normalfont\itshape}
  \theoremseparator{.}
  \theoremsymbol{}
  \newtheorem{#1}{#2}
}
\newsiamass{assumption}{Assumption}

\newsiamthm{example}{Example}
 
\usepackage{color} 

\headers{On the structure of reg. paths for piecewise diff. reg. terms}{B. Gebken, K. Bieker, and S. Peitz}

\title{On the structure of regularization paths for piecewise differentiable regularization terms\thanks{Submitted to the editors DATE.
\funding{This work was funded by the European Union and the German Federal State of North Rhine-Westphalia within the EFRE.NRW project ``SET CPS'', and by the DFG Priority Programme 1962 ``Non-smooth and Complementarity-based Distributed Parameter Systems''.}}}

\author{Bennet Gebken\thanks{Department of Mathematics, Paderborn University, 33098 Paderborn, Germany 
  (\email{bgebken@math.upb.de}, \email{bieker@math.upb.de}).}
\and Katharina Bieker\footnotemark[2]
\and Sebastian Peitz\thanks{Department of Computer Science, Paderborn University, 33098 Paderborn, Germany 
  (\email{sebastian.peitz@upb.de}).}
}

\usepackage{amsopn}

\ifpdf
\hypersetup{
  pdftitle={On the structure of regularization paths for piecewise differentiable regularization terms},
  pdfauthor={B. Gebken and K. Bieker and S. Peitz}
}
\fi

\begin{document}

\maketitle

\begin{abstract}
  Regularization is used in many different areas of optimization when solutions are sought which not only minimize a given function, but also possess a certain degree of regularity. Popular applications are image denoising, sparse regression and machine learning. Since the choice of the regularization parameter is crucial but often difficult, path-following methods are used to approximate the entire regularization path, i.e., the set of all possible solutions for all regularization parameters. Due to their nature, the development of these methods requires structural results about the regularization path. The goal of this article is to derive these results for the case of a smooth objective function which is penalized by a piecewise differentiable regularization term. We do this by treating regularization as a multiobjective optimization problem. Our results suggest that even in this general case, the regularization path is piecewise smooth. Moreover, our theory allows for a classification of the nonsmooth features that occur in between smooth parts. This is demonstrated in two applications, namely support-vector machines and exact penalty methods.
\end{abstract}



\section{Introduction}
    In optimization, \emph{regularization} is one of the basic tools for dealing with irregular solutions. For an objective function $f : \R^n \rightarrow \R$, the idea is to add a \emph{regularization term} $g : \R^n \rightarrow \R$ to $f$ which enforces regularity, and to weight $g$ with a \emph{regularization parameter} $\lambda \geq 0$ to control to which extent this regularity is enforced. So instead of optimizing $f$, the regularized problem
    \begin{align*}
        \min_{x \in \R^n} f(x) + \lambda g(x)
    \end{align*}
    with $\lambda \geq 0$ is solved. For $\lambda = 0$ the original problem is recovered. Increasing $\lambda$ leads to successively more regular solutions, at the cost of an increased objective value of $f$.  

    Depending on the application, the term ``regularity'' above can have many different meanings: In sparse regression, regularity of the solution means sparsity, and a prominent example for the regularization term is the $\ell^1$-norm \cite{T1996,HTF2009}. In hyperplane separation for data classification (also known as \emph{support-vector machines}), regularity is related to robustness of the derived classifier, and a possible regularization term can be derived from the scalar product of the data points with the hyperplane (known as the \emph{hinge loss}) \cite{B2006,HTF2009}. In image denoising, regularity means the absence of noise in the reconstructed image, which can be measured using the total variation \cite{C2004}. In (exact) penalty methods for constrained optimization problems, regularity refers to feasibility, and the sum of the individual constraint violations can be used as a regularization term \cite{NW2006,BKM2014}. Finally, in deep learning, regularization is used to avoid overfitting, which is related to the $\ell^2$- or $\ell^1$-norm of the weights \cite{B2006,GBC2016}.

    Clearly, the choice of the regularization parameter $\lambda$ has a large impact on the solution of the regularized problem. If $\lambda$ is chosen too small, then solutions are almost optimal for $f$ but irregular. If it is chosen too large, then solutions are highly regular but have an unacceptably large objective value with respect to $f$. One way of dealing with this issue is to not only compute a regularized solution for a single $\lambda$, but to compute the entire so-called \emph{regularization path} $R$, which is the set of all regularized solutions for all $\lambda \geq 0$. Obviously, simply solving the regularized problem for many $\lambda \geq 0$ to obtain a discretization of $R$ is inefficient. Instead, so-called \emph{path-following methods} (also known as \emph{continuation methods}, \emph{homotopy methods} or \emph{predictor-corrector methods}) can be used, which iteratively compute new points on the regularization path close to already known points until the complete path is explored. For the development of such methods, it is crucial to have a good understanding of the structure of the regularization path. In \cite{OPT2000,EHJ2004} it was shown that for sparse regression, the regularization path $R$ is piecewise linear and a path-following method was proposed for its computation. Similar results were shown in \cite{HRT2004} for support-vector machines. In a more general setting in \cite{RZ2007}, it was shown that if $f$ is piecewise quadratic and $g$ is piecewise linear, then $R$ is always piecewise linear. In case of the exact penalty method in constrained optimization, it was shown in \cite{ZL2015} that if the constrained problem is convex (and the equality constraints are affinely linear), then $R$ is piecewise smooth. Recently, in \cite{BGP2021}, the structure of the regularization path was analyzed for the case where $f$ is twice continuously differentiable and $g$ is the $\ell^1$-norm, with the results suggesting that $R$ is piecewise smooth.
    
    The goal of this article is to analyze the structure of the regularization path in a more general setting. Note that in the applications above, we have the pattern that $f$ is always smooth while $g$ is always nonsmooth. Thus, in this article, we will also assume that $f$ is smooth. For $g$, we will assume that it is merely \emph{piecewise differentiable} (as defined in \cite{S2012}). Compared to weaker assumptions in nonsmooth analysis like local Lipschitz continuity, this has the advantage that the Clarke subdifferential of $g$ is easy to compute and that the set of nonsmooth points of $g$ can essentially be described as a level set of certain smooth functions. Since all of the regularization terms in the above applications (except for the $\ell^2$-norm) are in fact piecewise differentiable, our setting generalizes many of the existing approaches. We will analyze the structure of $R$ by approximating it with the \emph{critical regularization path} $R_c$, which is based on the first-order optimality conditions of the regularized problem, and then identifying sufficient conditions for $R_c$ to be smooth around a given point. More precisely, our main result will be that if these conditions are met, then $R_c$ is locally the projection of a higher-dimensional smooth manifold onto $\R^n$ (cf.\ Theorem \ref{thm:h_level_set_manifold}). In particular, all points violating these conditions are potential ``kinks'' (or ``nonsmooth points'') of $R_c$. Depending on which condition is violated, this allows for a classification of nonsmooth features of the regularization path. Furthermore, the nature of our sufficient conditions suggests that $R_c$ (and $R$) is still piecewise smooth.
    
    The remainder of this article is structured as follows. In Section \ref{sec:basic_concepts}, we begin by introducing the basic concepts that we use in our theoretical results. Besides piecewise differentiability, these are \emph{multiobjective optimization} and \emph{affine geometry}. The former can be used to obtain an (almost) equivalent formulation of the regularization problem as a multiobjective optimization problem, while the latter is required for working with the subdifferential of $g$. In Section \ref{sec:structure_of_regularization_path}, we will analyze the structure of the regularization path $R$. We will do this by expressing $R_c$ as the union of the intersection of certain sets, whose structure we can analyze by applying standard results from differential geometry. In Section \ref{sec:examples}, we will apply our results to two problem classes, which are support-vector machines and the exact penalty method. Finally, we draw a conclusion and discuss possible future work in Section \ref{sec:conclusion}.
    
\section{Basic concepts} \label{sec:basic_concepts}
    
    \subsection{Piecewise differentiable functions}
        In the following, we will define piecewise differentiability and state the main results that we use throughout this article. For a more detailed introduction into the topic, we refer to \cite{S2012}. Let $U \subseteq \R^n$ be open.
    
        \begin{definition}
            Let $g : U \rightarrow \R$ be continuous and $g_i : U \rightarrow \R$, $i \in \{1,\dots,k\}$, be a set of $r$-times continuously differentiable (or $C^r$) functions for $r \in \N \cup \{\infty\}$. If $g(x) \in \{ g_1(x), \dots, g_k(x) \}$ for all $x \in U$, then $g$ is \emph{piecewise $r$-times differentiable} (or a $PC^r$\emph{-function}). In this case, $\{ g_1, \dots, g_k \}$ is called a \emph{set of selection functions} of $g$.
        \end{definition}
        
        When working with $PC^r$-functions in a local sense, it is useful to only consider the selection functions that have an impact on the local behavior around a given point. 
        
        \begin{definition}
            Let $g : U \rightarrow \R$ be a $PC^r$-function and let $\{ g_1, \dots, g_k \}$ be a set of selection functions of $g$. Then
            \begin{align*}
                I(x) := \{ i \in \{1,\dots,k\} : g(x) = g_i(x) \}
            \end{align*}
            is the \emph{active set} at $x \in U$. A selection function $g_i$ is called \emph{active at} $x$ if $i \in I(x)$.
        \end{definition}
    
        From the continuity of selection functions it follows that for any $x^0 \in \R^n$, there is an open neighborhood $U' \subseteq U$ of $x^0$ such that
        \begin{align} \label{eq:g_local_sel_functions}
            g(x) \in \{ g_i(x) : i \in I(x^0) \} \quad \forall x \in U'.
        \end{align}
        But note that not all active selection functions are necessarily required for the local representation of $g$ around $x^0$. For example, if a selection function is only active in $x^0$ and nowhere else, then it can be neglected. Hence, there is also the following, stricter definition of activity. To this end, for a set $A \subseteq \R^n$, we denote by $\cl(A)$ the \emph{closure} and by $\interior(A)$ the \emph{interior} of $A$ with respect to the natural topology on $\R^n$.
        
        \begin{definition}
            Let $g : U \rightarrow \R$ be a $PC^r$-function and let $\{ g_1, \dots, g_k \}$ be a set of selection functions of $g$. Then
            \begin{align*}
                I^e(x) := \{ i \in \{1,\dots,k\} : x \in \cl(\interior(\{ y \in U : g(y) = g_i(y)\})) \}
            \end{align*}
            is the \emph{essentially active set} at $x \in U$. A selection function $g_i$ is called \emph{essentially active at} $x$ if $i \in I^e(x)$.
        \end{definition}
        
        Due to continuity we have $I^e(x) \subseteq I(x)$ for all $x \in \R^n$. 
        By definition, if a selection function $g_i$ is essentially active at some point, then $\interior(\{ y \in U : g(y) = g_i(y)\})$ is non-empty. In other words, there is an open subset of $U$ on which $g$ behaves like $g_i$. The following lemma shows that locally, a given set of selection functions can always be reduced to those that are essentially active.
        \begin{lemma} \label{lem:local_ess_active}
            Let $g : U \rightarrow \R$ be a $PC^r$-function and let $\{ g_1, \dots, g_k \}$ be a set of selection functions of $g$. Then for any $x^0 \in U$, there is an open neighborhood $U' \subseteq U$ of $x^0$ such that $\{ g_i : i \in I^e(x^0) \}$ is a set of selection functions of the restriction $g|_{U'}$ of $g$ to $U'$.
        \end{lemma}
        \begin{proof}
            Proposition 2.22 in \cite{U2001}.
        \end{proof}
        
        Although we only assumed continuity in the definition of $PC^r$-functions, it is possible to show the following, stronger result. To this end, let $\| \cdot \|$ be any norm on $\R^n$.
        
        \begin{lemma} 
            Let $g : U \rightarrow \R$ be a $PC^r$-function. Then $g$ is locally Lipschitz continuous, i.e., for every $x \in U$, there is an open neighborhood $U' \subseteq U$ of $x$ and some $L > 0$ such that
            \begin{align*}
                | g(y) - g(z) | \leq L \| y - z \| \quad \forall y,z \in U'. 
            \end{align*}
        \end{lemma}
        \begin{proof}
            Corollary 4.1.1 in \cite{S2012}.
        \end{proof}
    
        While $PC^r$-functions are generally nonsmooth, the previous lemma allows us to use the so-called \emph{Clarke subdifferential} from nonsmooth analysis to obtain first-order approximations. To this end, for $A \subseteq \R^n$ let $\conv(A)$
        be the \emph{convex hull of} $A$.
        For a general locally Lipschitz continuous function $g : U \rightarrow \R$, let $\Omega \subseteq U$ be the set of points in which $g$ is not differentiable. Then the (Clarke) subdifferential of $g$ at $x \in U$ can be defined as
        \begin{align} \label{eq:subdiff}
            \partial g(x) := \conv \left( \left\{ \xi \in \R^n : \exists (x^j)_j \in \R^n \setminus \Omega \text{ with} \lim_{j \rightarrow \infty} x^j = x \text{ and} \lim_{j \rightarrow \infty} \nabla g(x^j) = \xi \right\} \right).
        \end{align} 
        If $g$ is continuously differentiable in $x$, then $\partial g(x) = \{ \nabla g(x) \}$. Although the subdifferential is a set, it behaves similarly to the standard derivative. For example, there are generalized versions of the chain rule and the mean-value theorem. Furthermore, as we will see later, it can be used to obtain optimality conditions. For a more detailed introduction into nonsmooth analysis, we refer to \cite{C1990, BKM2014}.
        
        In practice, computing the Clarke subdifferential of an arbitrary locally Lipschitz function can be difficult (cf.\ Section 3.3 in \cite{L1989}). But fortunately, for the special case of $PC^r$-functions, there is a simpler expression for the Clarke subdifferential in terms of the gradients of the selection functions. More precisely, we can use the following result.
    
        \begin{lemma} \label{lem:PC_subdiff}
            Let $g : U \rightarrow \R$ be a $PC^r$-function and let $\{ g_1, \dots, g_k \}$ be a set of selection functions of $g$. Then
            \begin{align} \label{eq:PC_subdiff}
                \partial g(x) = \conv(\{ \nabla g_i(x) : i \in I^e(x) \}) \quad \forall x \in \R^n.
            \end{align}
        \end{lemma}
        \begin{proof}
            Proposition 4.3.1 in \cite{S2012}.
        \end{proof}

        By the previous result, knowing the classical gradients of all essentially active selection functions is sufficient to obtain the exact Clarke subdifferential. In particular, since the number of selection functions is finite by definition, it follows that subdifferentials of $PC^r$-functions are always convex polytopes. 

        We conclude the introduction to $PC^r$-functions with some simple examples.
        \begin{example}
            \begin{itemize}
                \item[a)] Consider the $\ell^1$-norm on $\R^n$, i.e.,
                \begin{align*}
                    g : \R^n \rightarrow \R, \ x \mapsto \| x \|_1 := |x_1| + \dots + |x_n|.
                \end{align*}
                Then
                \begin{align*}
                    g(x) \in \left\{ \sum_{i = 1}^n \sigma_i x_i : \sigma \in \{-1,1\}^n \right\},
                \end{align*}
                so $g$ is $PC^\infty$ with selection functions
                \begin{align*}
                   g_{\sigma} : \R^n \rightarrow \R, \quad x \mapsto \sum_{i = 1}^n \sigma_i x_i \quad \text{for} \quad \sigma \in \{-1,1\}^n.
                \end{align*}
                For $x^0 \in \R^n$, the corresponding set of essentially active selection functions in $x^0$ is given by
                \begin{align*}
                    \left\{ g_\sigma : \sigma_i 
                    \begin{cases}
                        = \sign(x_i^0), & \text{if } x^0_i \neq 0 \\
                        \in \{-1,1\}, & \text{if } x^0_i = 0
                    \end{cases},
                    i \in \{1,\dots,n\}
                    \right\}.
                \end{align*}
                Therefore, the Clarke subdifferential of $g$ at $x^0 \in \R^n$ is given by
                \begin{align*}
                    \partial g(x^0) = \left\{ \xi \in \R^n : \xi_i
                    \begin{cases}
                        = \sign(x_i^0), & \text{if } x^0_i \neq 0 \\
                        \in [-1,1], & \text{if } x^0_i = 0
                    \end{cases},
                    i \in \{1,\dots,n\}
                    \right\}.
                \end{align*}
                \item[b)] As an example for a function that is differentiable almost everywhere but not $PC^r$ (for $r > 1$), consider
                \begin{align*}
                    g : \R \rightarrow \R, \ x \mapsto \sqrt{|x|}.
                \end{align*}
                Although $g$ is continuous everywhere and differentiable outside of $0$, it is not $PC^r$ (since it is not locally Lipschitz continuous in $0$).
            \end{itemize}
        \end{example}
    
    \subsection{Multiobjective optimization}
        In the following, we will give a brief introduction to (nonsmooth) multiobjective optimization. Due to the context of our paper, we only consider problems with two objectives here. For a more detailed introduction in the general case, we refer to \cite{M1998,E2005,MEK2014}.
        
        Let $f : \R^n \rightarrow \R$ and $g : \R^n \rightarrow \R$. The task of simultaneously minimizing $f$ and $g$ is denoted as
        \begin{align*} 
            \min_{x \in \R^n} 
            \begin{pmatrix}
                f(x) \\
                g(x)
            \end{pmatrix}
        \end{align*}
        and is called a \emph{multiobjective optimization problem} (MOP). The goal of multiobjective optimization is to find the so-called \emph{Pareto set}, which is defined as follows:
        
        \begin{definition}
            A point $x \in \R^n$ is called \emph{Pareto optimal} if there is no $y \in \R^n$ with 
            \begin{align*}
                f(y) < f(x) \text{ and } g(y) \leq g(x) \quad\quad \text{or} \quad\quad 
                f(y) \leq f(x) \text{ and } g(y) < g(x).
            \end{align*}
            The set of all Pareto optimal points is the \emph{Pareto set}. Its image under the objective vector $(f,g)$, i.e., the set $\{ (f(x),g(x))^\top : x \text{ is Pareto optimal} \} \subseteq \R^2$, is the \emph{Pareto front}.
        \end{definition}

         There are various different methods for solving nonsmooth MOPs, see, e.g., \cite{MKW2014,GP2021,DPA2002}. If both $f$ and $g$ are locally Lipschitz continuous, then the following theorem yields a necessary condition for Pareto optimality based on the Clarke subdifferentials (cf.\ \eqref{eq:subdiff}).
        
        \begin{theorem}
            Let $x \in \R^n$ be Pareto optimal. Then 
            \begin{align} \label{eq:KKT}
                0 \in \conv(\partial f(x) \cup \partial g(x)).
            \end{align}
        \end{theorem}
        \begin{proof}
            Theorem 12 in \cite{MEK2014}.
        \end{proof}
        
        Since \eqref{eq:KKT} is only a necessary condition, we make the following definition:
        \begin{definition}
            A point $x \in \R^n$ is called \emph{Pareto critical} if it satisfies \eqref{eq:KKT}. The set of all Pareto critical points is the \emph{Pareto critical set}, denoted by $P_c$.
        \end{definition}

        The Pareto critical set is a superset of the actual Pareto set. If $f$ and $g$ are convex, then \eqref{eq:KKT} is also sufficient, so the Pareto critical set coincides with the set of Pareto optimal points in that case (cf.\ Theorem 3.2.11 in \cite{M1998}).
        
        Using a result about the convex hull of the union of convex sets (cf.\ Lemma 5.29 in \cite{AB2006}), \eqref{eq:KKT} is equivalent to
        \begin{align} \label{eq:KKT_convex_comb}
            \exists \alpha_1, \alpha_2 \geq 0, \xi^1 \in \partial f(x), \xi^2 \in \partial g(x) : \alpha_1 \xi^1 + \alpha_2 \xi^2 = 0, \alpha_1 + \alpha_2 = 1.
        \end{align}
        In accordance with the smooth case, we will refer to such $\alpha_1$ and $\alpha_2$ as \emph{KKT multipliers of} $f$ \emph{and} $g$ \emph{in} $x$, respectively. Note that \eqref{eq:KKT_convex_comb} implies
        \begin{align*}
            0 \in \{ \lambda_1 \xi^1 + \lambda_2 \xi^2 : \lambda_1, \lambda_2 \in \R, \lambda_1 + \lambda_2 = 1 \},
        \end{align*}
        where the right-hand side is a so-called \emph{affine space}. The properties of such spaces are analyzed in the area of \emph{affine geometry}.
        
    \subsection{Affine geometry}
        In the following, we will introduce the basic concepts of affine geometry and affine spaces which we will use in this article. For further details on this topic, we refer to \cite{R1970,G2011,J2015}.

        \begin{definition}
            \begin{itemize}
                \item[a)] Let $k \in \N$ and $a^i \in \R^n$, $i \in \{1,\dots,k\}$. Let $\lambda \in \R^k$ with $\sum_{i = 1}^k \lambda_i = 1$. Then $\sum_{i = 1}^k \lambda_i a^i$ is an \emph{affine combination of} $\{ a^1, \dots, a^k \}$.
                \item[b)] Let $E \subseteq \R^n$. Then $\aff(E)$ is the set of all affine combinations of elements of $E$, called the \emph{affine hull of} $E$. Formally,
                \begin{align*}
                    \aff(E) := \left\{ \sum_{i = 1}^k \lambda_i a^i : k \in \N, a^i \in E, \lambda_i \in \R, i \in \{1,\dots,k\}, \sum_{i = 1}^k \lambda_i = 1 \right\}.
                \end{align*}
                \item[c)] Let $E \subseteq \R^n$. If $\aff(E) = E$, then $E$ is called an \emph{affine space}.
            \end{itemize}
        \end{definition}

        Affine spaces can be thought of as linear spaces that were translated away from the origin. More precisely, if $E$ is an affine space and $a' \in E$, then the set
        \begin{align} \label{eq:corr_linear_subspace}
            E - a' = \{ a - a' : a \in E \} =: V
        \end{align}
        is a linear subspace of $\R^n$. (Note that $V$ does not depend on the choice of $a'$). This allows for the definition of \emph{affine independence}.
        
        \begin{definition}
            Let $E$ be an affine space and let $a^i \in E$, $i \in \{1,\dots,k\}$. Then the set $\{ a^1, \dots, a^k \}$ is called \emph{affinely independent} if $\{ a^i - a^j : i \in \{1,\dots,k\} \setminus \{ j \} \}$ is linearly independent for some $j \in \{1,\dots,k\}$.
        \end{definition}

        If the condition in the previous definition holds for some $j \in \{1,\dots,k\}$, then it automatically holds for all $j \in \{1,\dots,k\}$ (cf.\ Lemma 2.4 in \cite{G2011}). As for linear independence, affine independence is related to uniqueness of the coefficients of affine combinations.
        
        \begin{lemma} \label{lem:affine_coeff_unique}
            Let $E$ be an affine space and let $a^i \in E$, $i \in \{1,\dots,k\}$. Let ${a \in \aff(\{ a^1, \dots, a^k \})}$ and $\lambda \in \R^k$ such that $a = \sum_{i = 1}^k \lambda_i a^i$ and $\sum_{i = 1}^k \lambda_i = 1$. Then $\{ a^1, \dots, a^k \}$ is affinely independent if and only if $\lambda$ is unique.
        \end{lemma}
        \begin{proof}
            Lemma 2.5 in \cite{G2011}.
        \end{proof}

        Furthermore, it is possible to assign a dimension to an affine space and define \emph{affine bases} (also known as \emph{affine frames}).
        \begin{definition}
            Let $k \in \N$ and let $E$ be an affine space with corresponding linear subspace $V$ (cf.\ \eqref{eq:corr_linear_subspace}).
            \begin{itemize}
                \item[a)] Let $a^i \in E$, $i \in \{1,\dots,k\}$. Then $\{ a^1, \dots, a^k \}$ is called an \emph{affine basis of} $E$ if $\{ a^i - a^j : i \in \{1,\dots,k\} \setminus \{ j \} \}$ is a basis of $V$ for some $j \in \{1,\dots,k\}$.
                \item[b)] The \emph{(affine) dimension} of $E$ is the dimension of $V$, denoted by $\affdim(E)$.
            \end{itemize}
        \end{definition}

        The previous definition implies that an affine basis consists of $\affdim(E) + 1$ elements.
        It is easy to show that if $\{ a^1, \dots, a^k \}$ forms an affine basis of $E$, then $\aff(\{ a^1, \dots, a^k \}) = E$ and combined with Lemma \ref{lem:affine_coeff_unique}, we obtain that every element in $E$ has a unique representation as an affine combination of $\{ a^1, \dots, a^k \}$.

        The main reason why we need affine geometry in this article is \emph{Carathéodory's theorem}, which gives us an upper bound for the number of elements we have to consider when computing convex hulls:
        \begin{theorem} \label{thm:caratheodory}
            Let $A$ be a finite subset of $\R^n$. Then every element in $\conv(A)$ can be written as a convex combination of $\affdim(\aff(A)) + 1$ elements of $A$.
        \end{theorem}
        \begin{proof}
            Theorem 3.1 in \cite{G2011}.
        \end{proof}

        Carathéodory's theorem will later be used to lower the number of selection functions we have to consider for the computation of the Clarke subdifferential $\partial g(x)$ of $g$.
        
        Finally, the concept of affine spaces enables us to introduce a useful definition of the interior and the boundary of ``low-dimensional'' subsets of $\R^n$.
        \begin{definition} \label{def:relative_interior}
            Let $A \subseteq \R^n$ and let $\aff(A)$ be endowed with the subspace topology of $\R^n$. Then the \emph{relative interior of} $A$, denoted by $\ri(A)$, is the interior of $A$ in $\aff(A)$, i.e.,
            \begin{align*}
                \ri(A) := \{ x \in A : \exists U \subseteq \R^n \text{ open with } x \in U \text{ and } U \cap \aff(A) \subseteq A \}.
            \end{align*}
            The \emph{relative boundary of} $A$ is the set $\rb(A) := \cl(A) \setminus \ri(A)$, where $\cl(A)$ is the closure of $A$ in $\R^n$.
        \end{definition}
        
        In the case of convex polytopes, i.e., convex hulls of a finite number of points, the relative interior and boundary can be expressed in terms of the coefficients of convex combinations:
        \begin{lemma} \label{lem:relint_polytope}
            Let $A = \conv(\{ a^1, \dots, a^k \}) \subseteq \R^n$ with $k \in \N$. Then 
            \begin{align*}
                \ri(A) &= \left\{ \sum_{i = 1}^k \lambda_i a^i : \lambda \in (\R^{>0})^k, \sum_{i = 1}^k \lambda_i = 1 \right\}, \\
                \rb(A) &= \left\{ a \in A : \nexists \lambda \in (\R^{>0})^k \text{ with } \sum_{i = 1}^k \lambda_i = 1 \text{ and } a = \sum_{i = 1}^k \lambda_i a^i \right\}.
            \end{align*}
        \end{lemma}
        \begin{proof}
            Exercise 3.1 in \cite{B1983}.
        \end{proof}

        For example, for a line connecting two points $a^1, a^2 \in \R^n$, $a^1 \neq a^2$, the relative boundary is the set containing $a^1$ and $a^2$ and the relative interior is the line without the end points, i.e., the set $\{ \lambda a^1 + (1-\lambda) a^2 : \lambda \in (0,1) \}$.

\section{The structure of the regularization path} \label{sec:structure_of_regularization_path}
    
    Let $f : \R^n \rightarrow \R$ be continuously differentiable and $g : \R^n \rightarrow \R$ be $PC^1$. For a \emph{regularization parameter} $\lambda \geq 0$, consider the parameter-dependent problem
    \begin{align} \label{eq:regularization_problem}
        \min_{x \in \R^n} f(x) + \lambda g(x).
    \end{align}
    The set
    \begin{align} \label{eq:regularization_path}
        R := \left\{ \bar{x} \in \R^n : \exists \lambda \geq 0 \ \text{with} \ \bar{x} \in \argmin_{x \in \R^n} f(x) + \lambda g(x) \right\}
    \end{align}
    is known as the \emph{regularization path} of \eqref{eq:regularization_problem} \cite{HRT2004,PH2007,MY2012} and the goal of this article is to analyze its structure. 
    
    We will do this by not analyzing $R$ directly, but by analyzing the (potentially larger) set that is defined by the first-order optimality condition of \eqref{eq:regularization_problem}: If $\bar{x}$ is a solution of \eqref{eq:regularization_problem} for some $\lambda \geq 0$, then it is a \emph{critical point} of $f + \lambda g$, i.e., $0 \in \partial (f + \lambda g)(\bar{x})$ (cf.\ Theorem 4.1 in \cite{BKM2014}). This is the motivation for defining the \emph{critical regularization path}
    \begin{align} \label{eq:critical_regularization_path}
        R_c := \left\{ \bar{x} \in \R^n : \exists \lambda \geq 0 \ \text{with} \ 0 \in \partial (f + \lambda g)(\bar{x}) \right\}.
    \end{align}   
    In general we have $R \subseteq R_c$. If $f + \lambda g$ is convex (e.g., if both $f$ and $g$ are convex), then criticality is sufficient for optimality (cf.\ Theorem 4.2 in \cite{BKM2014}), so $R = R_c$. For example, this is the case for the Lasso problem \cite{T1996} (where $f$ contains some least squares error and $g$ is the $\ell^1$-norm) and total variation denoising \cite{C2004} (where $f$ contains some least squares error and $g$ is the total variation).
    
    Our main result in this section will be that $R_c$ has a piecewise smooth structure. More precisely, we will derive five conditions (Assumptions \ref{assum:A1} to \ref{assum:A5}) for a point $x^0 \in R_c$ which, when combined, assure that locally around $x^0$, $R_c$ is the projection of a smooth manifold from a higher-dimensional space onto $\R^n$. In turn, these assumptions allow for a classification of kinks of $R_c$ by checking which assumption is violated. Throughout this article, we will use the term \emph{kinks} to loosely refer to points in $R_c$ around which $R_c$ is not a smooth manifold.
    
    In order to analyze the structure of $R_c$, we first show that $R_c$ is related to the Pareto critical set $P_c$ of the MOP
    \begin{align} \label{eq:MOP}
        \min_{x \in \R^n} 
        \begin{pmatrix}
            f(x) \\
            g(x)
        \end{pmatrix}.
    \end{align}
    More precisely, we have the following lemma.
    
    \begin{lemma} \label{lem:reg_path_critical_set}
        It holds:
        \begin{enumerate}
            \item[a)] $R_c = \{ \bar{x} \in \R^n : \exists \xi \in \partial g(\bar{x}), \alpha_1 > 0, \alpha_2 \geq 0 \ \text{with} \ \alpha_1 \nabla f(\bar{x}) + \alpha_2 \xi = 0 \ \text{and} \ \alpha_1 + \alpha_2 = 1 \} \subseteq P_c$.
            \item[b)] $R_c \cup \{ x \in \R^n : 0 \in \partial g(x) \} = P_c$.
        \end{enumerate}
    \end{lemma}
    \begin{proof}
        (a) Since $f$ is continuously differentiable we have $\partial f(x) = \{ \nabla f(x) \}$ for all $x \in \R^n$. Furthermore, from basic calculus for subdifferentials (cf.\ Corollary 1 in \cite{C1990}, Section 2.3) it follows that $\bar{x} \in R_c$ is equivalent to 
        \begin{equation} \label{eq:lem_reg_path_critical_set_1}
            \begin{aligned} 
                &\exists \lambda \geq 0 : 0 \in \partial (f + \lambda g)(\bar{x}) = \partial f(\bar{x}) + \lambda \partial g(\bar{x}) = \nabla f(\bar{x}) + \lambda \partial g(\bar{x}) \\
                \Leftrightarrow \ &\exists \lambda \geq 0 : 0 \in \frac{1}{1 + \lambda} \nabla f(\bar{x}) + \frac{\lambda}{1 + \lambda} \partial g(\bar{x}) \\
                \Leftrightarrow \ &\exists \xi \in \partial g(\bar{x}), \lambda \geq 0 : \frac{1}{1 + \lambda} \nabla f(\bar{x}) + \frac{\lambda}{1 + \lambda} \xi = 0 \\
                \Leftrightarrow \ &\exists \xi \in \partial g(\bar{x}), \alpha_1 > 0, \alpha_2 \geq 0 : \alpha_1 \nabla f(\bar{x}) + \alpha_2 \xi = 0 \ \text{and} \ \alpha_1 + \alpha_2 = 1.
            \end{aligned}
        \end{equation}
        By \eqref{eq:KKT_convex_comb} this implies $\bar{x} \in P_c$. \\
        (b) Due to (a) we only have to show the implication ``$\supseteq$'', so let $\bar{x} \in P_c$. By \eqref{eq:KKT_convex_comb} there are $\xi \in \partial g(\bar{x})$ and $\alpha_1, \alpha_2 \geq 0$ with $\alpha_1 + \alpha_2 = 1$ and $\alpha_1 \nabla f(\bar{x}) + \alpha_2 \xi = 0$. If $\alpha_1 = 0$ then $\alpha_2 = 1$, so $0 = \xi \in \partial g(\bar{x})$. Otherwise, $\alpha_1 > 0$ and from \eqref{eq:lem_reg_path_critical_set_1} it follows that $\bar{x} \in R_c$ (with $\lambda =\frac{\alpha_2}{\alpha_1}$).
    \end{proof}
    
    By the previous lemma, $R_c$ and $P_c$ coincide up to critical points of $g$ in which all KKT multipliers corresponding to $f$ are zero. Roughly speaking, these points correspond to ``$\lambda = \infty$'' in \eqref{eq:regularization_problem}. 
    
    \begin{remark} \label{rem:g_crit_R_c}
        It is important to note that Lemma \ref{lem:reg_path_critical_set} does not imply that critical points of $g$ are not contained in $R_c$, i.e., that $R_c \cap \{ x \in \R^n : 0 \in \partial g(x) \} = \emptyset$. For example, if $0 \in \interior(\partial g(x))$, then it is possible to show that there is some $\bar{\lambda}$ with $0 \in \partial (f + \lambda g)(x)$ for all $\lambda \geq \bar{\lambda}$.
    \end{remark}
    
    By Lemma \ref{lem:reg_path_critical_set}, structural results about Pareto critical sets can be used to analyze the structure of the critical regularization path $R_c$. For example, under some mild regularity assumptions on $f$ and $g$, Theorem 5.1 in \cite{H2001} shows that in areas where $g$ is (twice continuously) differentiable, the set of Pareto critical points with non-vanishing KKT multipliers is the projection of a $1$-dimensional manifold from $\R^{n+2}$ onto $\R^n$. To derive our main result, we will extend the ideas in \cite{H2001} to the whole Pareto critical set up to certain kinks.
    
    We begin by taking a closer look at the Pareto critical set $P_c$ of \eqref{eq:MOP}.
    By definition, $P_c$ is characterized by the optimality condition \eqref{eq:KKT}. Since $f$ is continuously differentiable and $g$ is $PC^1$, the subdifferential of $f$ is simply its gradient, and the subdifferential of $g$ is the convex hull of all essentially active selection functions (cf.\ Lemma \ref{lem:PC_subdiff}). Thus, for a fixed $x \in \R^n$, \eqref{eq:KKT} is equivalent to the existence of a vanishing convex combination of a finite number of elements. This is the same type of condition as in the smooth case, except that there is now no continuous dependency of these elements on $x$. Furthermore, the number of elements is not constant. Nonetheless, by iterating over all possible essentially active sets, $P_c$ can at least be written as the union of sets that behave similarly to Pareto critical sets in the smooth case. Let $\{ g_1, \dots, g_k \}$ be a set of selection functions of $g$. Then formally, these considerations lead to the following decomposition of $P_c$:
    \begin{align} \label{eq:Pc_for_fixed_I}
        P_c &= \{ x \in \R^n : 0 \in \conv(\{ \nabla f(x) \} \cup \partial g(x) ) \} \nonumber \\
        &= \{ x \in \R^n : 0 \in \conv(\{ \nabla f(x) \} \cup \{ \nabla g_i(x) : i \in I^e(x) \} ) \} \nonumber \\
        &= \bigcup_{I \subseteq \{1,\dots,k\}} P_c^I \cap \Omega^I,
    \end{align}
    where
    \begin{equation} \label{eq:def_PcI_OmegaI}
        \begin{aligned}
            P_c^I &:= \{ x \in \R^n : 0 \in \conv(\{ \nabla f(x) \} \cup \{ \nabla g_i(x) : i \in I \} ) \}, \\
            \Omega^I &:= \{ x \in \R^n : I^e(x) = I \}.
        \end{aligned}
    \end{equation}
    In words, $P_c^I$ is the Pareto critical set of the (smooth) MOP with objective vector $(f,g_{i_1}, \dots g_{i_{|I|}})^\top$ (for $I = \{i_1, ..., i_{|I|} \}$) and $\Omega^I$ is the set of points in $\R^n$ in which precisely the selection functions with an index in $I$ are essentially active. Thus, \eqref{eq:Pc_for_fixed_I} expresses $P_c$ as the union of Pareto critical sets of smooth MOPs that are intersected with the sets of points with constant essentially active sets. A visualization of this decomposition is shown in the following example.
    
    \begin{example} \label{example:simple_kink}
        Consider problem \eqref{eq:MOP} for $f : \R^2 \rightarrow \R$, $x \mapsto (x_1 - 2)^2 + (x_2 - 1)^2$, and
        \begin{align*}
            &g_1 : \R^2 \rightarrow \R, \quad x \mapsto x_1 + x_2, \\
            &g_2 : \R^2 \rightarrow \R, \quad x \mapsto x_1 - x_2, \\
            &g_3 : \R^2 \rightarrow \R, \quad x \mapsto -x_1 + x_2, \\
            &g_4 : \R^2 \rightarrow \R, \quad x \mapsto -x_1 - x_2, \\
            &g : \R^2 \rightarrow \R, \quad x \mapsto \max(\{g_1(x),g_2(x),g_3(x),g_4(x)\}) = \| x \|_1.
        \end{align*}
        It is possible to show that the Pareto critical (and in this case Pareto optimal) set is given by 
        \begin{align*}
            P_c &= \{ (0,0)^\top \} \cup ((0,1] \times \{ 0 \}) \cup \{ x \in \R^2: x_1 \in (1,2], x_2 = x_1 - 1 \} \\
            &= (P_c^{\{ 1,2,3,4 \}} \cap \Omega^{\{ 1,2,3,4 \}}) \cup (P_c^{\{ 1,2 \}} \cap \Omega^{\{ 1,2 \}}) \cup (P_c^{\{ 1 \}} \cap \Omega^{\{ 1 \}}) .
        \end{align*}
        Figure \ref{fig:simple_kink} shows the decomposition of $P_c$ into the sets $P_c^I \cap \Omega^I$ as in \eqref{eq:Pc_for_fixed_I}.
    	\begin{figure}
    		\centering
    		\includegraphics[width=0.55\textwidth]{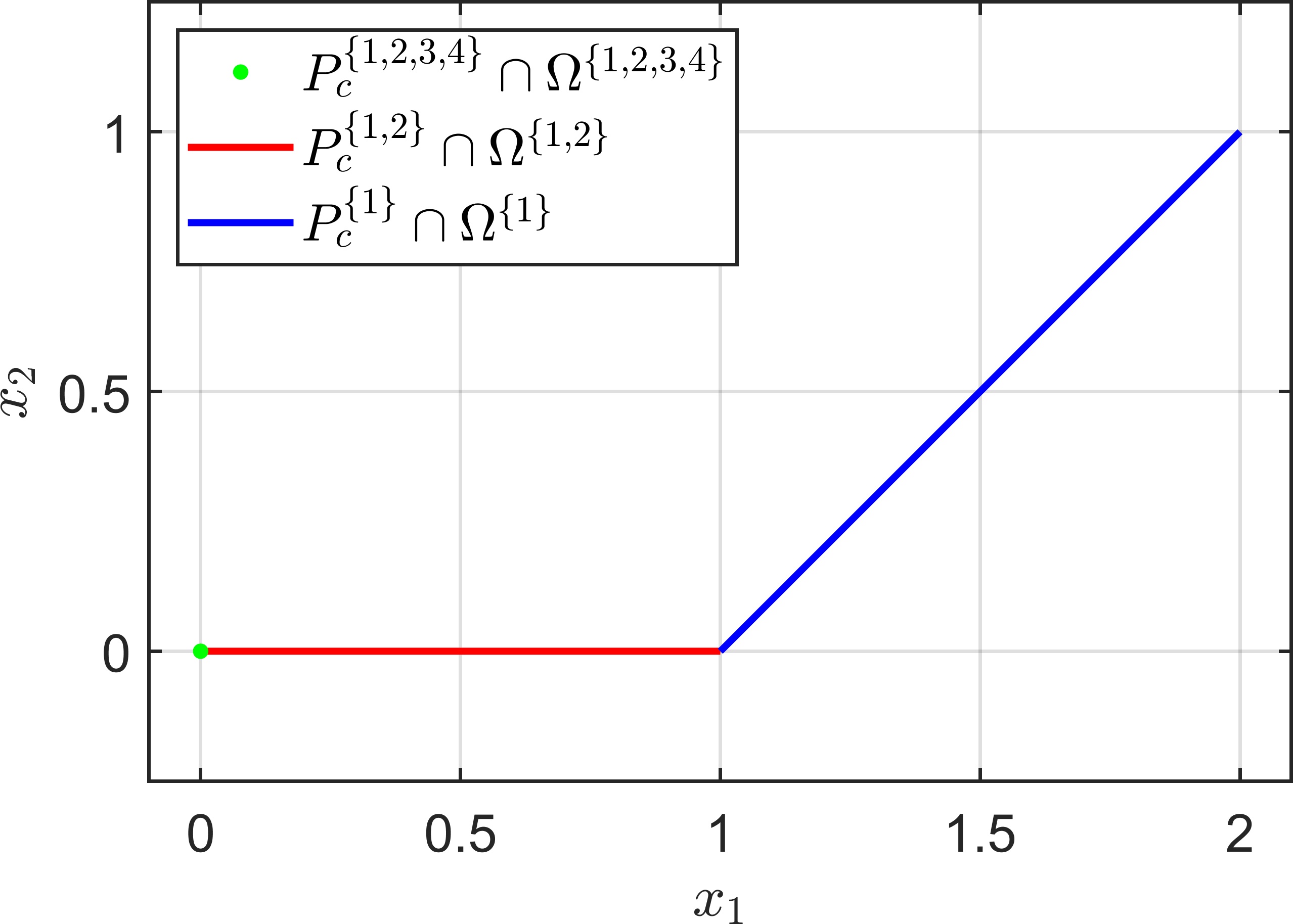}
    		\caption{Decomposition of $P_c$ into the sets $P_c^I \cap \Omega^I$ as in \eqref{eq:Pc_for_fixed_I}.}
    		\label{fig:simple_kink}
    	\end{figure}
    \end{example}
    
    We will analyze the piecewise smooth structure of $P_c$ via \eqref{eq:Pc_for_fixed_I} by first analyzing $\Omega^I$, then the intersection $P_c^I \cap \Omega^I$ and finally the union over all $P_c^I \cap \Omega^I$. Furthermore, as we expect $P_c$ to possess kinks, we will only consider its local structure around a given point. In other words, for $x^0 \in P_c$, we will only consider the structure of $P_c \cap U$ for open neighborhoods $U \subseteq \R^n$ of $x^0$.
    
    The strategy for our analysis in this section is to derive assumptions for $x^0$ which are sufficient for $P_c$ to have a smooth structure locally around $x^0$. These assumptions represent different sources and types of nonsmoothness of $P_c$ and will allow for a classification of nonsmooth points.
    
    \subsection{The structure of \texorpdfstring{$\Omega^I$}{OmegaI}}
        
        By definition, the set $\Omega^I$ only depends on $g$. For $I = \{i\} \subseteq \{1,...,k\}$, $\Omega^{\{ i \}}$ is the set of points where only the selection function $g_i$ is essentially active. From Lemma \ref{lem:local_ess_active} it follows that $\Omega^{\{ i \}}$ is an open subset of $\R^n$ in this case. For $I \subseteq \{1,...,k\}$ with $|I| > 1$, $\Omega^I$ is the set of points where precisely the selection functions corresponding to the elements of $I$ are essentially active. Typically (but not necessarily), these are points where $g$ is nonsmooth, which by Rademacher's Theorem (\cite{EG2015}, Theorem 3.2) form a null set. In the following, we will analyze its structure.
    
        Since we are only interested in the structure of $\Omega^I$ in a local sense, we also only have to consider restrictions $g|_U$ of $g$ to open neighborhoods of a point $x^0 \in \R^n$. 
        In terms of the open neighborhood $U$ of $x^0$ and the set of selection functions of $g|_U$, we introduce the following assumption:
        
        \begin{assumption} \label{assum:A1}
            For $x^0 \in \R^n$ there is an open neighborhood $U \subseteq \R^n$ of $x^0$ and a set of selection functions $\{ g_1, \dots, g_k \}$ of $g|_U$ such that            
            \begin{enumerate}
                \item[(i)] $I(x^0) = \{1,\dots,k\}$,
                \item[(ii)] $I^e(x) = I(x) \quad \forall x \in U$,
                \item[(iii)] $\affdim(\aff(\{\nabla g_i(x) : i \in \{1,\dots,k \} \}))\\ \hspace*{2cm} = \affdim(\aff(\{ \nabla g_i(x^0) : i \in \{1,\dots,k \} \})) \ \forall x \in U$.
            \end{enumerate}
        \end{assumption}
        
        Assumption \ref{assum:A1} can be interpreted as follows: \ref{assum:A1}(i) ensures that all selection functions we consider are actually relevant for the representation of $g$ in $U$. The condition \ref{assum:A1}(ii) ensures that it does not matter if we consider the active or the essentially active set in $U$, which allows for an easier representation of $\Omega^I$. Finally, \ref{assum:A1}(iii) makes sure that the representation of $\partial g(x^0)$ via the gradients of our selection functions is ``stable'' on $U$ with respect to its affine dimension.
        
        In the following, we will discuss the restrictiveness of Assumption \ref{assum:A1}. By \eqref{eq:g_local_sel_functions}, \ref{assum:A1}(i) can always be satisfied by choosing $U$ sufficiently small. For \ref{assum:A1}(ii) and (iii), we consider the following example.
        
        \begin{example} \label{example:structure_OmegaI}
            \begin{itemize}
                \item[a)] Let
                \begin{align*}
                    &g_1 : \R^2 \rightarrow \R, \quad x \mapsto x_2^2 - x_1, \\
                    &g_2 : \R^2 \rightarrow \R, \quad x \mapsto
                    \begin{cases}
                        x_1^2 - x_1, &x_1 \leq 0, \\
                        -x_1, &x_1 > 0,
                    \end{cases} \\
                    &g : \R^2 \rightarrow \R, \quad x \mapsto \max(\{ g_1(x), g_2(x) \}).
                \end{align*}
                Then $g$ is $PC^1$ with selection functions $g_1$ and $g_2$. The graph and the level sets of $g$ are shown in Figure \ref{fig:example_A1}. 
                \begin{figure}
    				\centering
    	            \subfloat[]{\includegraphics[width=0.475\textwidth]{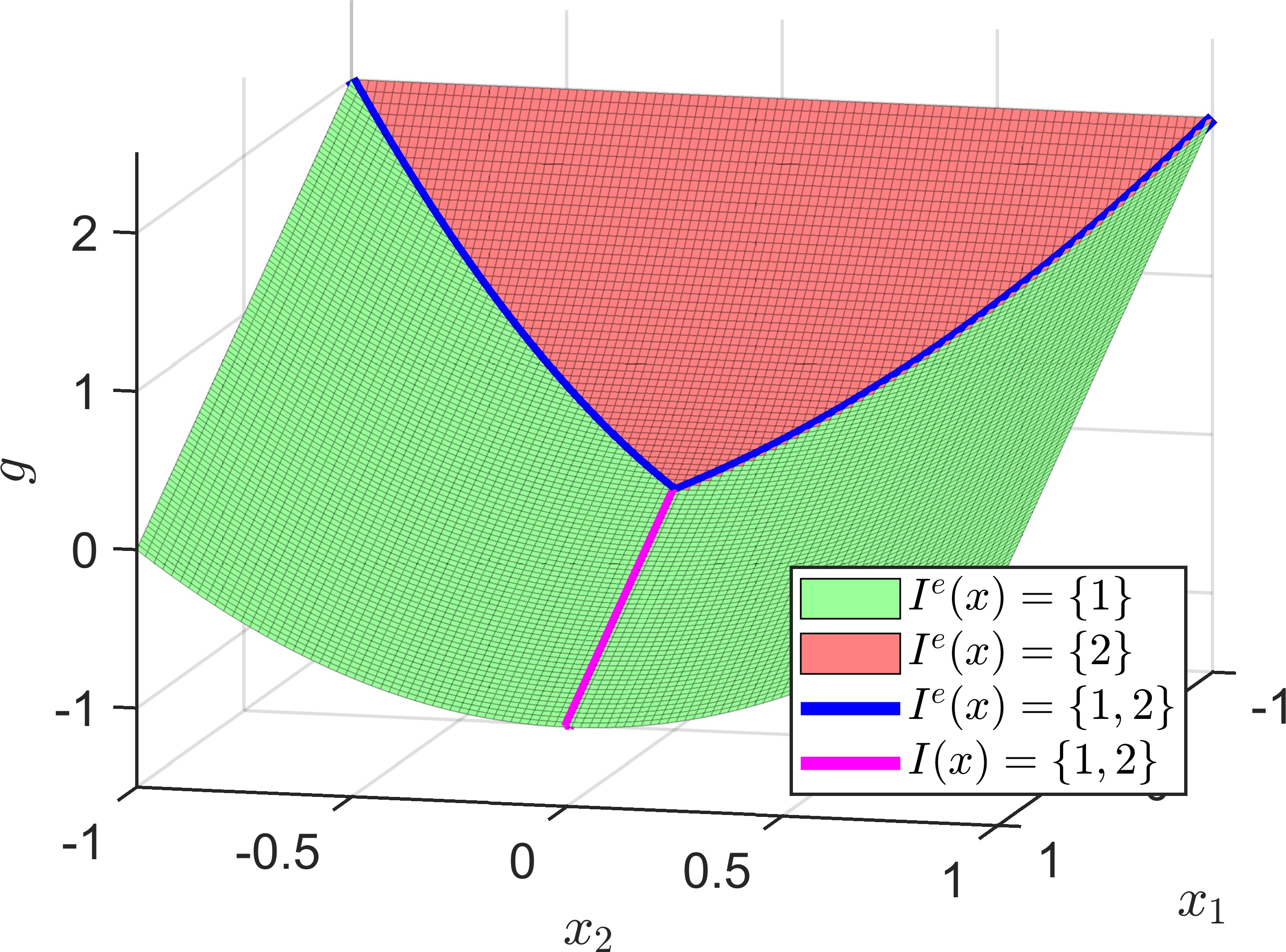}}
    	            \hfill
    	            \subfloat[]{\includegraphics[width=0.4\textwidth]{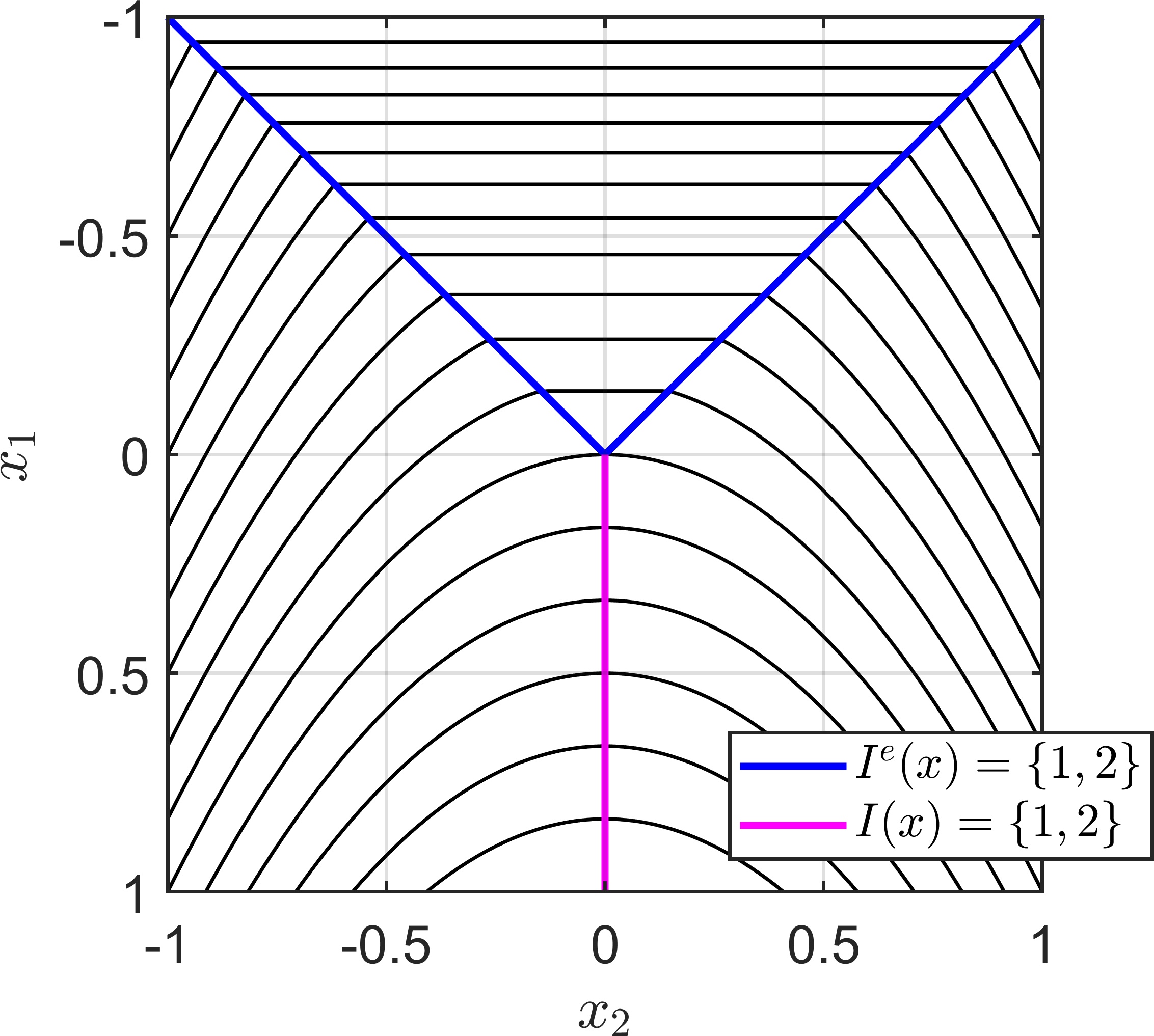}}
    				\caption{(a) The graph of the $PC^1$-function $g$ in Example \ref{example:structure_OmegaI} a). (b) The level sets of $g$.}
    				\label{fig:example_A1}
    			\end{figure}
                For the activity of $g_2$ we have
                \begin{align*}
                    2 \in I(x) \ \Leftrightarrow \ g(x) = g_2(x) \ \Leftrightarrow \ 
                    \begin{cases}
                        x_2 \in [x_1,-x_1], &x_1 \leq 0, \\
                        x_2 = 0, &x_1 > 0,
                    \end{cases}
                \end{align*}
                and
                \begin{align*}
                    2 \in I^e(x) \ &\Leftrightarrow \ x \in \cl(\interior(\{ y \in \R^2 : g(y) = g_2(y) \} )) \\ &\Leftrightarrow \  x_1 \leq 0,\ x_2 \in [x_1,-x_1].
                \end{align*}
                Thus, for any open neighborhood $U \subseteq \R^2$ of $x^0 = (0,0)^\top$, there is some $x \in U$ with $I^e(x) \neq I(x)$. In other words, \ref{assum:A1}(ii) does not hold in $x^0$ for this set of selection functions. But note that in this case, this can easily be fixed by modifying the behavior of $g_2$ for $x_1 > 0$. For example, replacing $g_2$ by
                \begin{align*}
                    \tilde{g}_2 : \R^2 \rightarrow \R, \quad x \mapsto
                    \begin{cases}
                        x_1^2 - x_1, &x_1 \leq 0, \\
                        -x_1^2 - x_1, &x_1 > 0.
                    \end{cases}
                \end{align*}
                solves the issue.
                \item[b)] For the selection functions $g_1$ and $\tilde{g}_2$ of $g$ as in a), we have 
                \begin{align*}
                    \nabla g_1(x) = 
                    \begin{pmatrix}
                        -1 \\
                        2 x_2
                    \end{pmatrix}
                    \ \text{and} \
                    \nabla \tilde{g}_2(x) = 
                    \begin{cases}
                        (2 x_1 - 1,0)^\top, &x_1 \leq 0, \\
                        (-2 x_1 - 1,0)^\top, &x_1 > 0.
                    \end{cases}
                \end{align*}
                In particular, in $x^0 = (0,0)^\top$ we have $\nabla g_1(x^0) = \nabla \tilde{g}_2(x^0) = (-1,0)^\top$, so 
                \begin{align*}
                     \affdim(\aff(\{ \nabla g_1(x^0), \nabla \tilde{g}_2(x^0) \})) = 0.
                \end{align*}
                But it is easy to see that
                \begin{align*}
                    \affdim(\aff(\{ \nabla g_1(x), \nabla \tilde{g}_2(x) \})) = 1 \quad \forall x \in \R^2 \setminus \{ 0\}.
                \end{align*}
                In particular, \ref{assum:A1}(iii) does not hold in $x^0$ (for this set of selection functions).
            \end{itemize}
        \end{example}
        
        By Lemma \ref{lem:local_ess_active}, for a given $x^0 \in \R^n$, we can always choose the open neighborhood $U$ of $x^0$ such that all selection functions of the local restriction $g|_U$ of $g$ are essentially active in $x^0$. In particular, we can assume that $I^e(x^0) = I(x^0)$. While this does not imply that (ii) holds in Assumption \ref{assum:A1}, the previous example shows how \ref{assum:A1}(ii) may be satisfied through modifications of the selection functions in areas where they are active, but not essentially active. Although we will not prove that this is always possible, it motivates us to believe that \ref{assum:A1}(ii) is not a strong assumption in practice.
        
        In contrast to \ref{assum:A1}(ii), modifying the selection functions will have less impact on \ref{assum:A1}(iii). The reason for this is the fact that if \ref{assum:A1}(i) and \ref{assum:A1}(ii) hold, then the right-hand side of \ref{assum:A1}(iii) is the dimension of the affine hull of the subdifferential of $g$ in $x^0$ (cf.\ Lemma \ref{lem:PC_subdiff}). In particular, the right-hand side does not depend on the choice of selection functions. In light of this, \ref{assum:A1}(iii) implies that the dimension of the affine hull of the subdifferential of $g$ is constant in all $x \in U$ with $I^e(x) = I^e(x^0)$, i.e., in all  $x \in \Omega^{I^e(x^0)}$ (cf.\ \eqref{eq:def_PcI_OmegaI}). Thus, \ref{assum:A1}(iii) is more related to the function $g$ and less related to the choice of selection functions. In Example \ref{example:structure_OmegaI} a), we see that the set $\Omega^{\{1,2\}}$ (in blue) has a kink in $x^0 = (0,0)^\top$. The following lemma suggests that this is caused by \ref{assum:A1}(iii) being violated. Thus, by assuming \ref{assum:A1}(iii), we limit ourselves to local restrictions $g|_U$ for which $\Omega^{I^e(x^0)}$ has a smooth structure.
        \begin{lemma} \label{lem:structure_OmegaI}
            Let $x^0 \in \R^n$. Let $U \subseteq \R^n$ be an open neighborhood of $x^0$ and let
            $\{ g_1, \dots, g_k \}$ be a set of selection functions of $g|_U$ as in Assumption \ref{assum:A1}. Let $d = \affdim(\aff(\partial g(x^0)))$ and let $\{ i_1, \dots, i_{d+1} \} \subseteq \{ 1, \dots, k \}$ such that $\{ \nabla g_i(x^0) : i \in \{ i_1, \dots, i_{d+1} \} \}$ is an affine basis of $\aff(\{ \nabla g_i(x^0) : i \in \{ 1, \dots, k \} \})$. Then there is an open neighborhood $U' \subseteq U$ of $x^0$ such that
            \begin{align*}
                g_i(x) - g_1(x) = 0 \ \forall i \in \{2,\dots,k\} \quad \Leftrightarrow \quad g_i(x) - g_{i_1}(x) = 0 \ \forall i \in \{i_2,\dots,i_{d+1}\}
            \end{align*}
            for all $x \in U'$ and $\Omega^{\{1,\dots,k\}} \cap U'$ is an embedded $(n-d)$-dimensional submanifold of $U'$. In particular,
            \begin{align*}
                \Omega^{\{1,\dots,k\}} \cap U' = \{ x \in U' : g_i(x) - g_{i_1}(x) = 0 \ \forall i \in \{i_2,\dots,i_{d+1}\} \}.
            \end{align*}
        \end{lemma}
        \begin{proof}
            The direction "$\Rightarrow$" is obvious, so consider the converse. By \ref{assum:A1}(iii) and since the gradients $\nabla g_i$, $i \in \{i_1,\dots,i_{d+1}\}$, are continuous, there is an open neighborhood $U' \subseteq U$ of $x^0$ such that $\{ \nabla g_i(x) : i \in \{i_1,\dots,i_{d+1}\} \}$ is an affine basis of $\{ \nabla g_i(x) : i \in \{1,\dots,k\} \}$ for all $x \in U'$. Let
            \begin{align*}
                \varphi : U' \rightarrow \R^{k-1}, \quad x \mapsto 
                \begin{pmatrix}
                    g_2(x) - g_1(x) \\
                    \vdots \\
                    g_k(x) - g_1(x)
                \end{pmatrix}.
            \end{align*}
            By \ref{assum:A1}(iii) the Jacobian $D\varphi(x)$ has constant rank $d$ for all $x \in U'$. By \ref{assum:A1}(i) we have $\varphi(x^0) = 0$, so the level set $L := \varphi^{-1}(0) = \Omega^{\{1,\dots,k\}} \cap U'$ is nonempty. Thus, by Theorem 5.12 in \cite{L2012}, $L$ is an embedded $(n-d)$-dimensional submanifold of $U'$. Additionally, let
            \begin{align*}
                \varphi' : U' \rightarrow \R^{d}, \quad x \mapsto 
                \begin{pmatrix}
                    g_{i_2}(x) - g_{i_1}(x) \\
                    \vdots \\
                    g_{i_{d+1}}(x) - g_{i_1}(x)
                \end{pmatrix}.
            \end{align*}
            By construction, $D\varphi'(x)$ has constant rank $d$ for all $x \in U'$. With the same argument as above, it follows that $L' := \varphi'^{-1}(0)$ is an embedded $(n-d)$-dimensional submanifold of $U'$ as well. Since $L \subseteq L'$, $L$ is also an embedded $(n-d)$-dimensional submanifold of $L'$ (cf.\ \cite{L2012}, Proposition 4.22). By Proposition 5.1 in \cite{L2012}, this implies that $L$ is an open subset of $L'$. As $L'$ is endowed with the subspace topology of $U' \subseteq \R^n$, this means that we can assume w.l.o.g.\ that $U'$ is an open neighborhood of $x^0$ with $U' \cap L' = L$, completing the proof.
        \end{proof}
    
        By the previous lemma, Assumption \ref{assum:A1} allows us to assume w.l.o.g.\ that for the restriction $g|_U$, the set of points with a constant active set $\Omega^{I^e(x^0)}$ is a smooth manifold around $x^0 \in U$ of dimension $n - \affdim(\aff(\partial g(x^0)))$. Furthermore, it shows that for the representation of $\Omega^{I^e(x^0)}$ as a level set, it is sufficient to only consider a subset of the set of selection functions whose gradients form an affine basis of $\partial g(x^0)$.
    
    \subsection{The structure of \texorpdfstring{$P_c^I \cap \Omega^I$}{PcIcapOmegaI}} \label{sec:structure_of_PcI_OmegaI}
    
        After analyzing the structure of $\Omega^I$, we will now turn towards the structure of the intersection $P_c^I \cap \Omega^I$ in \eqref{eq:Pc_for_fixed_I}.
        First of all, as for $\Omega^I$, we will show that not all selection functions of $g$ are required for the representation of $P_c^I \cap \Omega^I$. More precisely, a simple application of Carathéodory's theorem (Theorem \ref{thm:caratheodory}) to the definition of $P_c^I$ yields the following result.
        
        \begin{lemma} \label{lem:Pc_caratheodory}
            Let $x^0 \in P_c$ and let $\{g_1,\dots,g_k\}$ be a set of selection functions of $g$. If $x^0$ is not a critical point of $g$, then there is an index set $\{ i_1, \dots, i_r \} \subseteq \{ 1, \dots, k \}$ with $r = \affdim(\aff(\{ \nabla f(x^0) \} \cup \partial g(x^0)))$ such that
            \begin{enumerate}
                \item[a)] $0 \in \conv(\{ \nabla f(x^0) \} \cup \{ \nabla g_i(x^0) : i \in \{i_1,\dots,i_r \}\})$,
                \item[b)] $\{ \nabla f(x^0) \} \cup \{ \nabla g_i(x^0) : i \in \{i_1,\dots,i_r \}\}$ is affinely independent.
            \end{enumerate}
        \end{lemma}
        \begin{proof}
            By Theorem \ref{thm:caratheodory}, there is an affinely independent subset of
            \begin{align*}
                \{ \nabla f(x^0) \} \cup \{ \nabla g_i(x^0) : i \in \{1,\dots,k\} \}
            \end{align*}
            of size $r+1$ with zero in its convex hull. Since $x^0$ is not a critical point of $g$, $\nabla f(x^0)$ must be contained in that subset.
        \end{proof}
        
        With Lemma \ref{lem:structure_OmegaI} and Lemma \ref{lem:Pc_caratheodory}, we have ways to simplify $\Omega^I$ and $P_c^I$, respectively, by only considering certain selection functions of $g$. But note that we can not necessarily choose the same selection functions for both results: Although the set $\{ \nabla g_i(x^0) : i \in \{i_1,\dots,i_r\}\}$ in Lemma \ref{lem:Pc_caratheodory} is affinely independent, the index set $\{i_1,\dots,i_r\}$ can not necessarily be used in Lemma \ref{lem:structure_OmegaI} since we might have $r < d+1$, i.e.,
        \begin{equation} \label{eq:derivation_A2}
            \begin{aligned}
                &\affdim(\aff(\{ \nabla f(x^0) \} \cup \partial g(x^0))) < \affdim(\aff(\partial g(x^0))) + 1 \\
                \Leftrightarrow \ &\aff(\{ \nabla f(x^0) \} \cup \partial g(x^0)) = \aff(\partial g(x^0)) \\
                \Leftrightarrow \ &\nabla f(x^0) \in \aff(\partial g(x^0)).
            \end{aligned}      
        \end{equation}
        In particular, since $x^0$ is Pareto critical, this would imply that $0 \in \aff(\partial g(x^0))$ (even though $x^0$ is not critical for $g$, i.e., $0 \notin \conv(\partial g(x^0))$). The following lemma shows that this scenario is related to the uniqueness of the KKT multiplier corresponding to $f$ in $x^0$.
        
        \begin{lemma} \label{lem:alpha1_unique}
            Let $x^0 \in P_c$ such that $x^0$ is not a critical point of $g$.
            \begin{itemize}
                \item[a)] If the KKT multiplier $\alpha_1$ of $f$ in $x^0$ (cf.\ \eqref{eq:KKT_convex_comb}) is not unique, then $\nabla f(x^0) \in \aff(\partial g(x^0))$.
                \item[b)] If $\nabla f(x^0) \in \aff(\partial g(x^0))$ and $0$ is contained in the relative interior (cf.\ Definition \ref{def:relative_interior}) of $\conv(\{ \nabla f(x^0) \} \cup \partial g(x^0))$, then the KKT multiplier $\alpha_1$ of $f$ in $x^0$ is not unique.
            \end{itemize}
        \end{lemma}
        \begin{proof}
            See \ref{SM:proof_lem_affine_coeff_unique} in the supplementary material.
        \end{proof}
        
        \begin{remark} \label{rem:kink_front}
            In \cite{H2001}, Section 4.3, it was shown that in the smooth case and under certain regularity assumptions on $f$ and $g$, the coefficient vector of the vanishing convex combination in the KKT condition in a point $x \in P_c$, i.e., the vector $(\alpha_1, \alpha_2)^\top$ in \eqref{eq:KKT_convex_comb}, is orthogonal to the tangent space of the image of the Pareto critical set at $(f(x),g(x))^\top$. Thus, roughly speaking, non-uniqueness of $(\alpha_1, \alpha_2)^\top$ suggests that this tangent space is ``degenarate'', i.e., that the Pareto front possesses a kink at $(f(x),g(x))^\top$.
        \end{remark}
        
        The following example shows what behavior may occur if the KKT multiplier of $f$ is not unique.
        \begin{example} \label{example:zero_in_aff_g}
            Consider problem \eqref{eq:MOP} for $f : \R^2 \rightarrow \R$, $x \mapsto x_1^2 + x_2^2$, and
            \begin{align*}
                &g_1 : \R^2 \rightarrow \R, \quad x \mapsto x_1^2 + (x_2 - 1)^2, \\
                &g_2 : \R^2 \rightarrow \R, \quad x \mapsto x_1^2 + (x_2 - 1)^2 - \left( x_2 - \frac{1}{2} \right), \\
                &g : \R^2 \rightarrow \R, \quad x \mapsto \max(\{g_1(x),g_2(x)\}).
            \end{align*}
            Then $g$ is $PC^1$ with selection functions $g_1$ and $g_2$. It is easy to see that 
            \begin{align*}
                \Omega^{\{1,2\}} = \{ x \in \R^n : I^e(x) = \{1,2\} \} = \R \times \left\{ \frac{1}{2} \right\},
            \end{align*}
            as depicted in Figure \ref{fig:zero_in_aff_g}(a).
    		\begin{figure}[ht] 
    			\centering 
    			\subfloat[]{\includegraphics[width=0.45\textwidth]{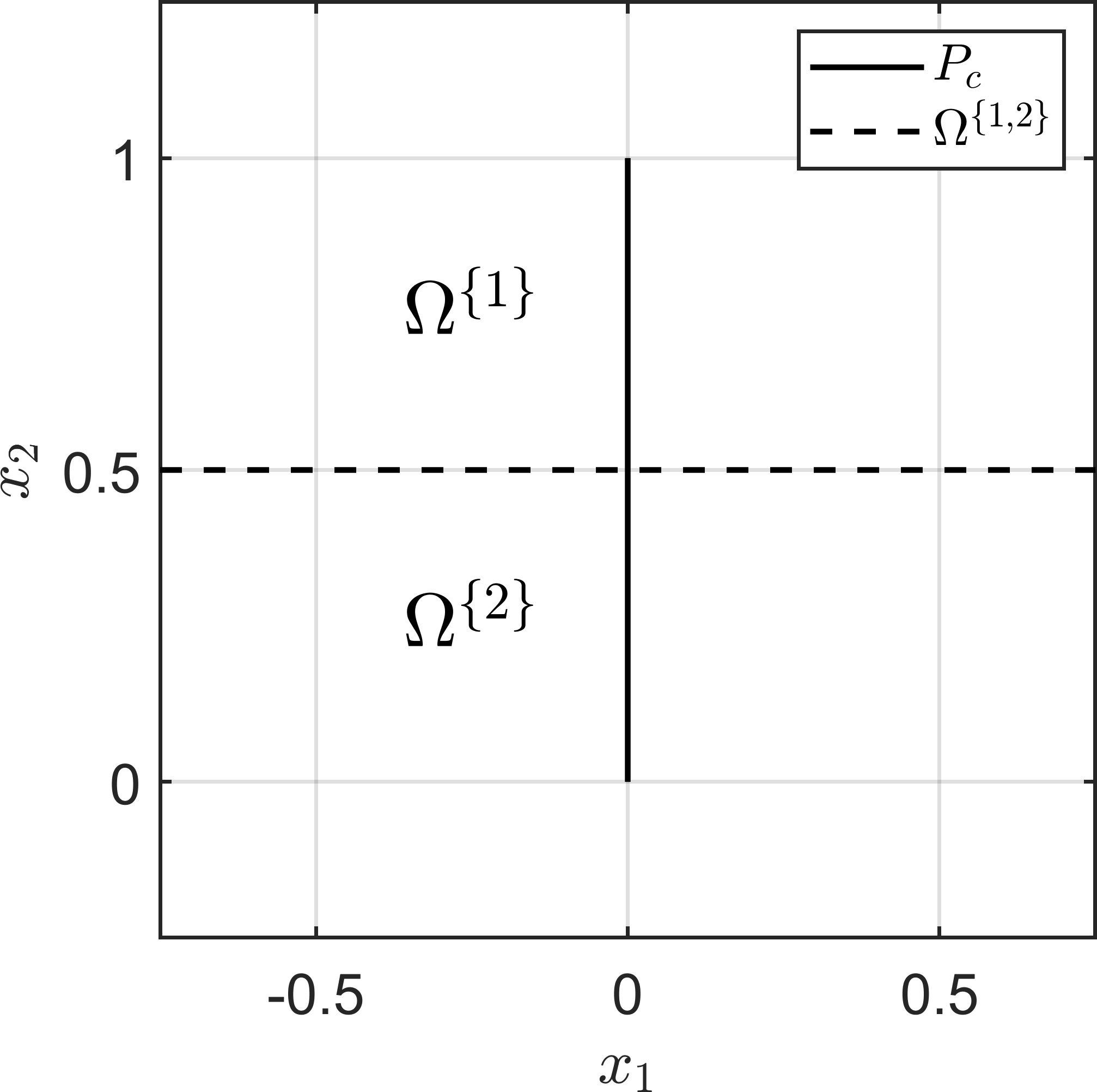}}
    			\hfill
    	        \subfloat[]{\includegraphics[width=0.45\textwidth]{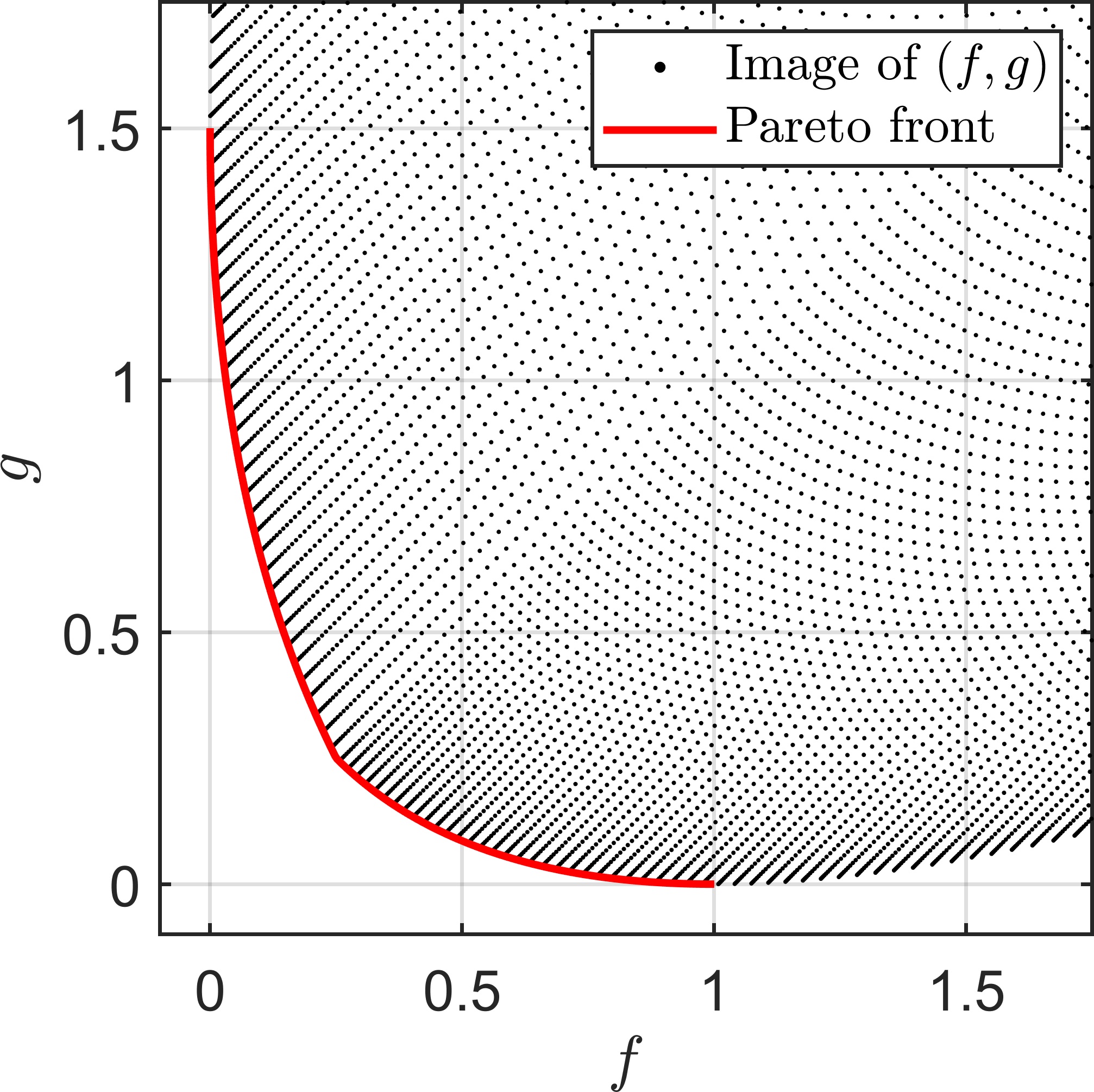}}
    			\caption{(a) Pareto critical set $P_c$ and $\Omega^I$, $I \subseteq \{1,2\}$, in Example \ref{example:zero_in_aff_g}. (b) Pointwise discretization of the image $\{ (f(x),g(x))^\top : x \in \R^2 \}$ of the objective vector $(f,g)$ and the image of the Pareto critical set under $(f,g)$.}
    			\label{fig:zero_in_aff_g}
    		\end{figure}
    		The Pareto critical (and in this case Pareto optimal) set is given by $P_c = \{ 0 \} \times [0,1]$. In particular, $x^0 = (0,\frac{1}{2})^\top$ is the only Pareto critical point where more than one selection function is active, i.e., $P_c^{\{1,2\}} \cap \Omega^{{\{1,2\}}} = \{ x^0 \}$. By computing the gradients in $x^0$, we obtain
    		\begin{align*}
    		    \nabla f(x^0) = (0,1)^\top, \ \nabla g_1(x^0) = (0,-1)^\top, \ \nabla g_2(x^0) = (0,-2)^\top.
    		\end{align*}
        	We see that
    		\begin{align*}
    		    \frac{1}{2} \nabla f(x^0) + \frac{1}{2} \nabla g_1(x^0) = 0 \quad \text{and} \quad \frac{2}{3} \nabla f(x^0) + \frac{1}{3} \nabla g_2(x^0) = 0,
    		\end{align*}
    		so the KKT multiplier of $f$ is not unique. By Lemma \ref{lem:alpha1_unique} this implies $\nabla f(x^0) \in \aff(\{ \partial g(x^0) \})$. More explicitly, for this example, it is easy to check that
    		\begin{align*}
    		    \nabla f(x^0) = 3 \nabla g_1(x^0) - 2 \nabla g_2(x^0).
    		\end{align*}
    		Figure \ref{fig:zero_in_aff_g}(b) shows an approximation of the image of $(f,g)$ and the image of the Pareto critical set. As discussed in Remark \ref{rem:kink_front}, we see that the image of $P_c$ has a kink at $(f(x^0),g(x^0))^\top = (\frac{1}{4},\frac{1}{4})^\top$.
        \end{example}
        
        As the previous example suggests, a scenario where the KKT multiplier of $f$ is not unique may occur if the Pareto critical set goes transversally through the set of nonsmooth points instead of moving tangentially along it. In other words, it may occur if arbitrarily close to $x^0 \in P_c$, there are Pareto critical points with essentially active sets $I_1$ and $I_2$ such that $I_1 \neq I_2$ and $I_1 \neq I^e(x^0) \neq I_2$. Due to continuity of the gradients, the KKT multipliers for both sets $I_1$ and $I_2$ have accumulation points that are KKT multipliers of $x^0$. Since $I_1 \neq I_2$, these accumulation points may not coincide, such that the KKT multipliers in $x^0$ are not unique. In terms of the structure of $P_c^I \cap \Omega^I$, we see that it is a $0$-dimensional set in Example \ref{example:zero_in_aff_g} (for $I = \{1,2\}$) as it is just a single point. 
             
        Although Pareto critical points $x^0$ with $\nabla f(x^0) \in \aff(\partial g(x^0))$ may not necessarily cause nonsmoothness of $P_c$, we will still exclude them from our consideration of the local structure of $P_c$ around $x^0$ to avoid the irregularities discussed above. So formally, we introduce the following assumption:
        \begin{assumption} \label{assum:A2}
            For $x^0 \in P_c$ we have
            \begin{align*}
                \nabla f(x^0) \notin \aff(\partial g(x^0)).
            \end{align*}
        \end{assumption}
    
        Roughly speaking, since $\affdim(\aff(\partial g(x^0))) < n$ in most cases, we expect that the set of points that violate Assumption \ref{assum:A2} is small compared to $P_c$ (or even empty). By \eqref{eq:derivation_A2}, Assumption \ref{assum:A2} implies that there is an index set as in  Lemma \ref{lem:Pc_caratheodory} that satisfies the requirements of Lemma \ref{lem:structure_OmegaI}. In particular, $P_c^I \cap \Omega^I$ can then be expressed using only a subset of the selection functions of $g$.
        
        The discussion of $P_c^I \cap \Omega^I$ so far was mainly focused on the removal of redundant information in the subdifferential of $g$ to simplify our analysis.
        We will now turn towards its actual geometrical structure. To this end, we again consider Example \ref{example:simple_kink}.
        
        \begin{example} \label{example:simple_kink_2}
            Let $f$ and $g$ be as in Example \ref{example:simple_kink}. (The corresponding Pareto critical set is shown in Figure \ref{fig:simple_kink}.) Let $x^0 = (1,0)^\top$ and $U \subseteq \R^2$ be the open ball with radius one around $x^0$. Then a set of selection functions of $g|_U$ is given by $\{g_1, g_2\}$ and we have $P_c^{\{ 1,2 \}} \cap \Omega^{\{ 1,2 \}} = (0,1] \times \{ 0 \}$. In particular, $x^0$ is a boundary point of $P_c^{\{ 1,2 \}} \cap \Omega^{\{ 1,2 \}}$, such that $P_c^{\{ 1,2 \}} \cap \Omega^{\{ 1,2 \}}$ is not smooth around $x^0$ (in the sense of smooth manifolds). The gradients of $f$, $g_1$ and $g_2$ are shown in Figure \ref{fig:simple_kink_2}. 
            \begin{figure}
    		    \centering
    		    \includegraphics[width=0.45\textwidth]{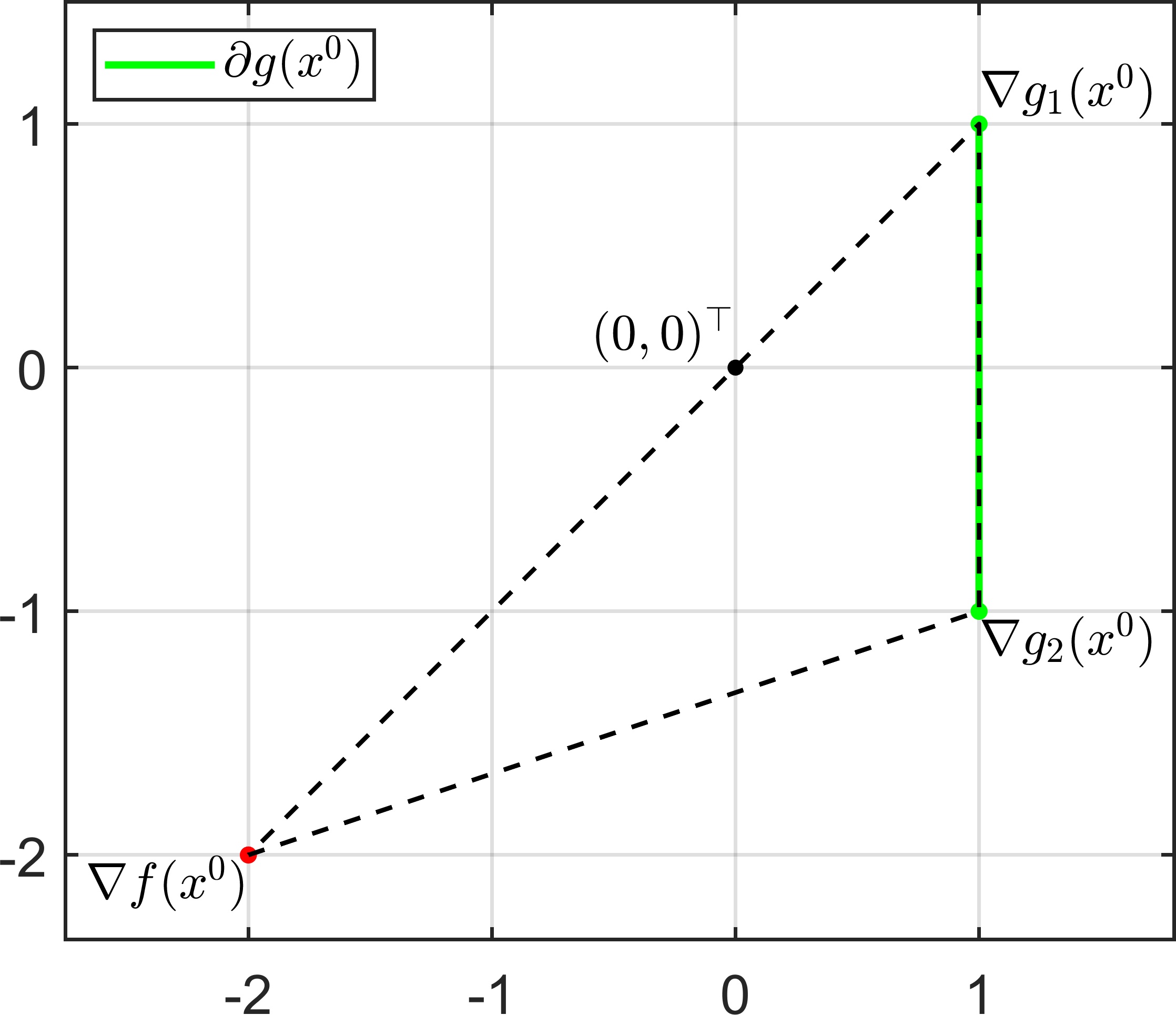}
    		    \caption{The gradients of $f$, $g_1$ and $g_2$ in $x^0 = (1,0)^\top$ in Example \ref{example:simple_kink_2}. The dashed line shows the (relative) boundary of the convex hull $\conv(\{ \nabla f(x^0) \} \cup \partial g(x^0))$.}
    		    \label{fig:simple_kink_2}
    	    \end{figure}
            We see that there is a unique convex combination
            \begin{align} \label{eq:example_simple_kink_coeffs}
                \frac{1}{3} \nabla f(x^0) + \frac{2}{3} \nabla g_1(x^0) + 0 \nabla g_2(x^0) = 0
            \end{align}
            where the coefficient of $\nabla g_2(x^0)$ is zero.
        \end{example}
    
        Note that in the previous example, there is still a vanishing affine combination of the gradients of $f$, $g_1$ and $g_2$ for $x = (x_1,0)^\top$, $x_1 > 1$. But it is not a convex combination, as the coefficient corresponding to $\nabla g_2(x)$ is negative. Due to the  continuity of the gradients, this can only happen if one of the coefficients in $x^0$ is already zero (as in \eqref{eq:example_simple_kink_coeffs}). To exclude the type of nonsmoothness caused by this, we introduce the following assumption.
        \begin{assumption} \label{assum:A3}
            For $x^0 \in P_c$ and a set of selection functions $\{ g_1, \dots, g_k \}$ of $g$, there is an index set $\{ i_1, \dots, i_r \} \subseteq \{1,\dots,k\}$ as in Lemma \ref{lem:Pc_caratheodory} and positive coefficients $\alpha^0 > 0$, $\beta^0 \in (\R^{>0})^r$ with $\alpha^0 + \sum_{j = 1}^r \beta^0_j = 1$ and $\alpha^0 \nabla f(x^0) + \sum_{j = 1}^r \beta^0_j \nabla g_{i_j}(x^0) = 0$.
        \end{assumption}
    
        The following lemma yields a necessary condition for Assumption \ref{assum:A3} to hold, which is related to the relative interior (cf.\ Definition \ref{def:relative_interior}) of $\conv(\{ \nabla f(x^0) \} \cup \partial g(x^0))$. In particular, it is independent of the choice of selection functions.
    
        \begin{lemma} \label{lem:zero_in_relative_int}
            Let $x^0 \in P_c$. If there is a set of selection functions such that Assumption \ref{assum:A3} holds, then 
            \begin{align*}
                0 \in \ri(\conv(\{ \nabla f(x^0) \} \cup \partial g(x^0))).
            \end{align*}
        \end{lemma}
        \begin{proof}
            See \ref{SM:proof_lem_zero_in_relative_int} in the supplementary material.
        \end{proof}
        
        After introducing the Assumptions \ref{assum:A1}, \ref{assum:A2} and \ref{assum:A3}, we are now able to show the first structural result about $P_c^I \cap \Omega^I$. The following lemma shows that $P_c^I \cap \Omega^I$ is the projection of a level set from a higher-dimensional space onto the variable space $\R^n$.
        \begin{lemma} \label{lem:intersec_level_set}
            Let $x^0 \in P_c$. Let $U \subseteq \R^n$ be an open neighborhood of $x^0$ and let $\{ g_1, \dots, g_k \}$ be a set of selection functions of $g|_U$ satisfying Assumptions \ref{assum:A1} and \ref{assum:A3}. Assume that Assumption \ref{assum:A2} holds. Then there is an index set $\{ i_1, \dots, i_r \} \subseteq \{1,\dots,k\}$ and an open neighborhood $U' \subseteq U$ of $x^0$ such that
            \begin{align} \label{eq:PcI_OmegaI_as_projection}
                P_c^{\{1,\dots,k\}} \cap \Omega^{\{1,\dots,k\}} \cap U' = {\pr}_x(h^{-1}(0)) \cap U',
            \end{align}
            where ${\pr}_x : \R^n \times \R \times \R^r \rightarrow \R^n$ is the projection onto the first $n$ components and
            \begin{align*}
                h: \R^n \! \times \! \R^{>0} \! \times \! (\R^{>0})^r \rightarrow \R^n \! \times \! \R \! \times \! \R^{r-1},
                (x,\alpha,\beta) \mapsto 
                \begin{pmatrix}
                    \alpha \nabla f(x) + \sum_{j = 1}^r \beta_j \nabla g_{i_j}(x) \\
                    \alpha + \sum_{j = 1}^r \beta_j - 1 \\
                    (g_{i_j}(x) - g_{i_1}(x))_{j \in \{2,\dots,r\}}
                \end{pmatrix}.
            \end{align*}
        \end{lemma}
        \begin{proof}
            Let $\{ i_1, \dots, i_r \} \subseteq \{1,\dots,k\}$ be an index set as in \ref{assum:A3}. Since the gradients $\nabla f$ and $\nabla g_{i_j}$, $j \in \{1,\dots,r\}$, are continuous and $\{ \nabla f(x^0) \} \cup \{ \nabla g_{i_j}(x^0) : j \in \{1,\dots,r\} \}$ is affinely independent, there is an open neighborhood $U' \subseteq U$ of $x^0$ such that $\{ \nabla f(x) \} \cup \{ \nabla g_{i_j}(x) : j \in \{1,\dots,r\} \}$ is affinely independent for all $x \in U'$. In particular,
            \begin{equation} \label{eq:lem_intersec_level_set_1}
                \begin{aligned}
                    r &\leq \affdim(\aff(\{ \nabla f(x) \} \cup \{ \nabla g_i(x) : i \in \{1,\dots,k\})) \\
                    &\leq \affdim(\aff(\{ \nabla g_i(x) : i \in \{1,\dots,k\})) + 1 \quad \forall x \in U'.
                \end{aligned}
            \end{equation}
            By \ref{assum:A1}, \ref{assum:A2} and \ref{assum:A3}, we have 
            \begin{equation} \label{eq:lem_intersec_level_set_2}
                \begin{aligned}
                    r &\stackrel{\ref{assum:A3}}{=} \affdim(\aff(\{ \nabla f(x^0) \} \cup \partial g(x^0)))
                    \stackrel{\ref{assum:A2}}{=} \affdim(\aff(\partial g(x^0))) + 1 \\
                    &\stackrel{\ref{assum:A1}(i),(ii)}{=} \affdim(\aff(\{ \nabla g_i(x^0) : i \in \{1,\dots,k\})) + 1 \\
                    &\stackrel{\ref{assum:A1}(iii)}{=} \affdim(\aff(\{ \nabla g_i(x) : i \in \{1,\dots,k\})) + 1 \quad \forall x \in U'.
                \end{aligned}
            \end{equation}
            Combining \eqref{eq:lem_intersec_level_set_1} and \eqref{eq:lem_intersec_level_set_2}, we obtain
            \begin{align*}
                \affdim(\aff(\{ \nabla f(x) \} \cup \{ \nabla g_i(x) : i \in \{1,\dots,k\})) = r \quad \forall x \in U',
            \end{align*}
            so $\{ \nabla f(x) \} \cup \{ \nabla g_{i_j}(x) : j \in \{1,\dots,r\} \}$ is an affine basis of $\{ \nabla f(x) \} \cup \{ \nabla g_i(x) : i \in \{1,\dots,k\} \}$ for all $x \in U'$. \\
            Let $x \in P_c^{\{1,\dots,k\}} \cap \Omega^{\{1,\dots,k\}} \cap U'$. By Lemma \ref{lem:affine_coeff_unique}, every element of $\aff(\{ \nabla f(x) \} \cup \{ \nabla g_i(x) : i \in \{1,\dots,k\} \})$ can be uniquely written as an affine combination of elements of $\{ \nabla f(x) \} \cup \{ \nabla g_{i_j}(x) : j \in \{1,\dots,r\} \}$. Let $\alpha^0$ and $\beta^0$ as in \ref{assum:A3}. Since $\alpha^0 > 0$, $\beta^0 \in (\R^{>0})^r$ and the gradients $\nabla f$, $\nabla g_{i_j}$, $j \in \{1,\dots,r\}$, are continuous, we can assume w.l.o.g.\ that $U'$ is small enough such that there are $\alpha > 0$, $\beta \in (\R^{>0})^r$ with $\alpha + \sum_{j = 1}^r \beta_j = 1$ and
            \begin{align*}
                \alpha \nabla f(x) + \sum_{j = 1}^r \beta_j \nabla g_{i_j}(x) = 0.
            \end{align*}
            Furthermore, $g_{i_j}(x) - g_{i_1}(x) = 0$ holds for all $j \in \{2,\dots,r\}$ since $x \in \Omega^{\{1,\dots,k\}}$. Thus, $h(x,\alpha,\beta) = 0$, i.e., $x \in {\pr}_x(h^{-1}(0)) \cap U'$. \\
            Now let $x \in {\pr}_x(h^{-1}(0)) \cap U'$. Then $x \in P_c^{\{1,\dots,k\}}$ trivially holds since $\{ i_1, \dots, i_r \} \subseteq \{1,\dots,k\}$. By \ref{assum:A1} and Lemma \ref{lem:structure_OmegaI}, we can assume w.l.o.g.\ that $U'$ is small enough such that $g_{i_j}(x) - g_{i_1}(x) = 0$ for all $j \in \{2,\dots,r\}$ implies $x \in \Omega^{\{1,\dots,k\}}$, completing the proof.
        \end{proof}
        
        Up to this point, we assumed $f$ to be continuously differentiable and $g$ to be $PC^1$. This means that the map $h$ in the previous lemma is at least continuous. If $h$ is actually continuously differentiable, then standard results from differential geometry can be used to analyze the structure of its level sets on the right-hand side of \eqref{eq:PcI_OmegaI_as_projection}. To this end, we will assume for the remainder of this section that $f$ is twice continuously differentiable and $g$ is $PC^2$. 
        
        \begin{theorem} \label{thm:h_level_set_manifold}
            In the setting of Lemma \ref{lem:intersec_level_set} it holds:
            \begin{itemize}
                \item[a)] If $Dh(x,\alpha,\beta)$ has full rank for all $(x,\alpha,\beta) \in h^{-1}(0)$, then $h^{-1}(0)$ is a $1$-dimensional submanifold of $\R^n \times \R^{>0} \times (\R^{>0})^r$.
                \item[b)] If $Dh(x,\alpha,\beta)$ has constant rank $m \in \N$ for all $(x,\alpha,\beta) \in \R^n \times \R^{>0} \times (\R^{>0})^r$, then $h^{-1}(0)$ is an $(n+r+1-m)$-dimensional submanifold of $\R^n \times \R^{>0} \times (\R^{>0})^r$. 
            \end{itemize}
            In both cases, the tangent space of $h^{-1}(0)$ is given by
            \begin{align} \label{eq:tangent_space}
                T_{(x,\alpha,\beta)} (h^{-1}(0)) = \ker(Dh(x,\alpha,\beta)).
            \end{align}
        \end{theorem}
        \begin{proof}
            Part a) follows from Corollary 5.14 and part b) follows from Theorem 5.12 in \cite{L2012}. The formula for the tangent space follows from Proposition 5.38 in \cite{L2012}.
        \end{proof}
        
        \begin{remark}
            Equation \eqref{eq:tangent_space} in the previous theorem can be used to compute tangent vectors of the regularization path in practice by computing elements of  ${\pr}_x(\ker(Dh(x,\alpha,\beta)))$. Thus, it is an essential result for the construction of path-following methods.
        \end{remark}
        
        The previous theorem is the main result in this section. It shows that the structure of $h^{-1}(0)$ (and thus the structure of $P_c^I \cap \Omega^I$ due to \eqref{eq:PcI_OmegaI_as_projection}) is related to the rank of the Jacobian $Dh$, given by
        \begin{align*}
            \begin{pmatrix}
                \alpha \nabla^2 f(x) + \sum_{j = 1}^r \beta_j \nabla^2 g_{i_j}(x) & \nabla f(x) & \nabla g_{i_1}(x) & \hdots & \nabla g_{i_r}(x) \\
                0 & 1 & 1 & \hdots & 1 \\
                (\nabla g_{i_2}(x) - \nabla g_{i_1}(x))^\top & 0 & 0 & \hdots & 0 \\
                \vdots & \vdots & \vdots &  & \vdots \\
                (\nabla g_{i_r}(x) - \nabla g_{i_1}(x))^\top & 0 & 0 & \hdots & 0
            \end{pmatrix}
            \in \R^{(n+r) \times (n+r+1)}
        \end{align*}
        for $(x,\alpha,\beta) \in \R^n \times \R^{>0} \times (\R^{>0})^r$. Note that in Theorem \ref{thm:h_level_set_manifold} b), the assumption on the rank has to hold for all $(x,\alpha,\beta) \in \R^n \times \R^{>0}\times (\R^{>0})^{r}$ whereas in a), it only has to hold for all $(x,\alpha,\beta) \in h^{-1}(0)$. The following remark shows how the structure of $Dh$ can be used to analyze its rank.
    
        \begin{remark} \label{rem:rankDh}
            In the setting of Lemma \ref{lem:intersec_level_set}, let $(v^x,v^\alpha,v^\beta) \in \ker(Dh(x,\alpha,\beta)) \subseteq \R^n \times \R^{>0} \times (\R^{>0})^r$, i.e.,
            \begin{equation} \label{eq:kerDh_explicit}
                \begin{aligned}
                    &\left( \alpha \nabla^2 f(x) + \sum_{j = 1}^r \beta_j \nabla^2 g_{i_j}(x) \right) v^x + v^\alpha \nabla f(x) + \sum_{j = 1}^r v^\beta_j \nabla g_{i_j}(x) = 0, \\
                    &v^\alpha + \sum_{j = 1}^r v^\beta_j = 0, \\
                    &(\nabla g_{i_j}(x) - \nabla g_{i_1}(x))^\top v^x = 0 \quad \forall j \in \{2,\dots,r\}.
                \end{aligned}
            \end{equation}
            Since $\{ \nabla f(x), \nabla g_{i_1}(x), \dots, \nabla g_{i_r}(x) \}$ is affinely independent by construction (cf.\ proof of Lemma~\ref{lem:intersec_level_set}), the set
            \begin{align*}
                W := \left\{ v^\alpha \nabla f(x) + \sum_{j = 1}^r v^\beta_j \nabla g_{i_j}(x) : v^\alpha \in \R, v^\beta \in \R^r, v^\alpha + \sum_{j = 1}^r v^\beta_j = 0 \right\}
            \end{align*}
            is an $r$-dimensional linear subspace of $\R^n$. Similar to Lemma \ref{lem:affine_coeff_unique}, it is possible to show that for each element of $W$, the corresponding coefficients $v^\alpha$ and $v^\beta$ are unique. If $\alpha \nabla^2 f(x) + \sum_{j = 1}^r \beta_j \nabla^2 g_{i_j}(x)$ is regular, then the first two lines of \eqref{eq:kerDh_explicit} are equivalent to
            \begin{align*}
                v^x \in -\left(\alpha \nabla^2 f(x) + \sum_{j = 1}^r \beta_j \nabla^2 g_{i_j}(x)\right)^{-1} W =: V_1,
            \end{align*}
            where $V_1$ is an $r$-dimensional linear subspace of $\R^n$. In particular, $v^\alpha$ and $v^\beta$ are uniquely determined by $v^x$. Furthermore, if we denote by $V^\bot$ the orthogonal complement of a subspace $V$, then the last line of \eqref{eq:kerDh_explicit} is equivalent to 
            \begin{align*}
                v^x \in \spn(\{ \nabla g_{i_j}(x) - \nabla g_{i_1}(x) : j \in \{2,\dots,r\} \})^\bot =: V_2,
            \end{align*}
            where $V_2$ is an $(n - (r - 1))$-dimensional subspace of $\R^n$ since $\{ \nabla g_{i_1}(x), \dots, \nabla g_{i_r}(x) \}$ is affinely independent. Thus, the dimension of $\ker(Dh(x,\alpha,\beta))$ is given by the dimension of the intersection $V_1 \cap V_2$. 
            If we assume that $V_1$ and $V_2$ are generic subspaces, then we can apply a basic result from linear algebra to see that
            \begin{align*}
                \dim(\ker(Dh(x,\alpha,\beta))) &= \dim(V_1 \cap V_2) = \dim(V_1) + \dim(V_2) - \dim(V_1 + V_2) \\
                &= r + (n - (r - 1)) - n = 1,
            \end{align*}
            i.e., the rank of $Dh(x,\alpha,\beta)$ is full and Theorem \ref{thm:h_level_set_manifold} a) can be applied.
        \end{remark}
        
        The previous remark suggests that $h^{-1}(0)$ is typically a $1$-dimensional manifold such that we expect $P_c^I \cap \Omega^I$ to be ``$1$-dimensional'' as well by \eqref{eq:PcI_OmegaI_as_projection}. Nonetheless, we will see later that there are applications where $h^{-1}(0)$ is a higher-dimensional manifold. Furthermore, there are cases where $h^{-1}(0)$ is not a manifold at all. (Note that this is not necessarily caused by the nonsmoothness of $g$ and can also happen for smooth objective functions (cf.\ Example 1 in \cite{GPD2019}).) Thus, for $P_c^I \cap \Omega^I$ to have a smooth structure around a (corresponding) $x^0 \in P_c$, we have to make the following assumption:
        
        \begin{assumption} \label{assum:A4}
            In the setting of Lemma \ref{lem:intersec_level_set}, Theorem \ref{thm:h_level_set_manifold} can be applied, i.e.,
            \begin{enumerate}
                \item[(a)] $\rk(Dh(x,\alpha,\beta)) = n + r \quad \forall (x,\alpha,\beta) \in h^{-1}(0)$ or
                \item[(b)] $\rk(Dh(x,\alpha,\beta)) \text{ is constant } \quad \forall (x,\alpha,\beta) \in \R^n \times \R^{>0} \times (\R^{>0})^r.$
            \end{enumerate}
        \end{assumption}
        
        We conclude the discussion of the structure of $P_c^I \cap \Omega^I$ by considering the special case where $f$ is quadratic and $g$ is piecewise (affinely) linear. Remark \ref{SM:rem_quadratic_piecewise_linear} in the supplementary material shows that in this case, $P_c^I \cap \Omega^I$ is (locally) an affinely linear set around points that satisfy the assumptions of Lemma \ref{lem:intersec_level_set}. This coincides with the results in \cite{RZ2007}.

    \subsection{The structure of \texorpdfstring{$P_c$}{Pc}}
    
        After analyzing the structure of $P_c^I \cap \Omega^I$, we are now in the position to analyze the structure of the Pareto critical set $P_c$ of \eqref{eq:MOP}. By \eqref{eq:Pc_for_fixed_I}, $P_c$ can be written as the union of $P_c^I \cap \Omega^I$ for all possible combinations $I$ of selection functions. Since we already discussed the structure of the individual $P_c^I \cap \Omega^I$, the only additional nonsmooth points in $P_c$ may arise by taking their union. More precisely, nonsmooth points may arise where the different $P_c^I \cap \Omega^I$ touch, i.e., where the set of (essentially) active selection functions changes. The following lemma yields a necessary condition for identifying such points.
        
        \begin{lemma} \label{lem:PcIOmegaI_touch}
            Let $x^0 \in P_c$ and let $\{g_1, \dots, g_k\}$ be a set of selection functions of $g$ with $I^e(x^0) = \{ i_1, \dots, i_l \}$, $l \in \N$. If for all open neighborhoods $U \subseteq \R^n$ of $x^0$, there is some $x \in P_c \cap U$ with $I^e(x) \neq I^e(x^0)$, then there are $\alpha \geq 0$ and $\beta \in (\R^{\geq 0})^{l}$ such that $\alpha + \sum_{j = 1}^{l} \beta_j = 1$, 
            \begin{align*}
                \alpha \nabla f(x^0) + \sum_{j = 1}^{l} \beta_j \nabla g_{i_j}(x^0) = 0
            \end{align*}
            and $\beta_j = 0$ for some $j \in \{1,\dots,l\}$.
        \end{lemma}
        \begin{proof}
            See \ref{SM:proof_lem_PcIOmegaI_touch} in the supplementary material.
        \end{proof}
        
        A visualization of the previous lemma can be seen in Example \ref{example:simple_kink}: In $x^0 = (1,0)^\top$, the sets $P_c^{\{1,2\}} \cap \Omega^{\{1,2\}}$ and $P_c^{\{1\}} \cap \Omega^{\{1\}}$ touch and there is a convex combination with a zero component (cf.\ \eqref{eq:example_simple_kink_coeffs}). In this case, this causes a kink in $P_c$.
        
        Note that in general, the existence of a coefficient vector with a zero component as in Lemma \ref{lem:PcIOmegaI_touch} is not a useful criterion to find points in $P_c$ where the active set changes. For example, by Lemma \ref{lem:Pc_caratheodory}, if the number of essentially active selection functions in $x^0$ is larger than $\affdim(\aff(\{ \nabla f(x^0) \} \cup \partial g(x^0)))$, then there is always a coefficient vector with a zero component. A stricter condition would be that \emph{every} coefficient vector has a zero component, i.e., that zero is located on the relative boundary of $\conv(\{ \nabla f(x^0) \} \cup \partial g(x^0))$ (cf.\ Definition~\ref{def:relative_interior}). By Lemma \ref{lem:zero_in_relative_int}, this would imply that Assumption \ref{assum:A3} cannot hold, such that $P_c^I \cap \Omega^I$ may be nonsmooth around $x^0$. Although the theory suggests (and we will later explicitly see this in Example \ref{ex:Exact_Penalty_Ex2}) that this must not necessarily be the case in points where the active set changes, we believe it may be a useful criterion in practice. 
        
        Nonetheless, from a theoretical point of view, the only reliable assumption we can make to exclude points where the essentially active set changes is the following: 
        
        \begin{assumption} \label{assum:A5}
            For $x^0 \in P_c$ and a set of selection functions $\{ g_1, \dots, g_k \}$ of $g$, there is an open neighborhood $U \subseteq \R^n$ of $x^0$ such that
            \begin{align*}
                I^e(x) = I^e(x^0) \quad \forall x \in P_c \cap U.
            \end{align*}
        \end{assumption}

        \begin{table} 
        \caption{An overview of the five assumptions required to have a smooth structure of $P_c$ around $x^0 \in P_c$.}
        \begin{center}
        \begin{tabular}{|c |l|}
            \hline
            \multicolumn{2}{|l|}{}\\[-1.0em]
            \multicolumn{2}{|l|}{Let $x^0 \in P_c$.\hfill}\\
            \multicolumn{2}{|l|}{}\\[-1.0em]
            \hline
             & \\[-1.0em]
            \ref{assum:A1} & There is an open nbd. $U \ni x^0$ and a set of sel. fct. $\{ g_1, \dots, g_k \}$ of $g|_U$ with\\  
            & ~~\textit{(i)}  $\quad I(x^0) = \{1,\dots,k\}$,\\
            &  ~\textit{(ii)} $\quad I^e(x) = I(x) \quad \forall x \in U$,\\
            & \textit{(iii)} $\quad\affdim(\aff(\{\nabla g_i(x) : i \in \{1,\dots,k \} \}))$  \\
            & $\hspace{32pt} =\affdim(\aff(\{ \nabla g_i(x^0) : i \in \{1,\dots,k \} \})) \ \forall x \in U$.\\
             & \\[-1.0em]
            \hline
             & \\[-1.0em]
            \ref{assum:A2}  & It holds $\nabla f(x^0) \notin \aff(\partial g(x^0))$.\\
             & \\[-1.0em]
            \hline
             & \\[-1.0em]
            \ref{assum:A3} & Let $\{ g_1, \dots, g_k \}$ be a set of selection functions of $g$. \\
            & It exists $\{ i_1, \dots, i_r \} \subseteq \{1,\dots,k\}$ and $\alpha^0 \in \R$, $\beta^0 \in \R^r$\\ 
            & with $\alpha^0 + \sum_{j = 1}^r \beta^0_j = 1$ such that \\
            &$\left.\begin{array}{l} 
            \hspace{3pt} \textit{(i)} \quad  r = \affdim(\aff(\{ \nabla f(x^0) \} \cup \partial g(x^0))), \\
            \textit{(ii)} \quad \{ \nabla f(x^0) \} \cup \{ \nabla g_i(x^0) : i \in \{i_1,\dots,i_r \}\} \text{ aff. ind.},
            \end{array} \right\}$ (cf.\ Lemma~\ref{lem:Pc_caratheodory})\\
            & \hspace{-0.15pt} \textit{(iii)} \hspace{3pt} $\alpha^0 \nabla f(x^0) + \sum_{j = 1}^r \beta^0_j \nabla g_{i_j}(x^0) = 0$, \\
            & \hspace{1.4pt} \textit{(iv)} \hspace{3pt} $\alpha^0 > 0$, $(\beta^0)_j >0$.\\
             & \\[-1.0em]
            \hline
             & \\[-1.0em]
            \ref{assum:A4} & Assume that \ref{assum:A1}, \ref{assum:A2} and \ref{assum:A3} hold and
            let $h$ be defined as in Lemma~\ref{lem:intersec_level_set}.\\
             & \textit{(a)} $\quad \rk(Dh(x,\alpha,\beta)) = n + r \quad \forall (x,\alpha,\beta) \in h^{-1}(0)$ \textbf{or}\\
             & \textit{(b)} $\quad \rk(Dh(x,\alpha,\beta)) \text{ is constant } \quad \forall (x,\alpha,\beta) \in \R^n \times \R^{>0} \times (\R^{>0})^r$.\\
              & \\[-1.0em]
            \hline
             & \\[-1.0em]
            \ref{assum:A5} &  Let $\{ g_1, \dots, g_k \}$ be a set of selection functions of $g$.\\ 
            & There is an open neighborhood $U \ni x^0$ with $I^e(x) = I^e(x^0) \quad \forall x \in P_c \cap U$.\\
            \hline
        \end{tabular} \label{tab:assum_overview}
        \end{center}
        \end{table}
        
        From our considerations up to this point it follows that if $x^0 \in P_c$ is a point in which Assumptions \ref{assum:A1} to \ref{assum:A5} hold (for the same set of selection functions), then $P_c$ is the projection of a smooth manifold around $x^0$ as in Theorem \ref{thm:h_level_set_manifold}. An overview of all five assumptions is shown in Table \ref{tab:assum_overview}. Unfortunately, in contrast to Assumptions \ref{assum:A1}, \ref{assum:A2}, \ref{assum:A3} and \ref{assum:A4}, \ref{assum:A5} is only an a posteriori condition, i.e., we already have to know $P_c$ around $x^0$ to be able to check if Assumption \ref{assum:A5} holds.

        \begin{remark}
            If Assumption \ref{assum:A5} is violated in $x^0 \in P_c$, then there are Pareto critical points arbitrarily close to $x^0$ with a different (essentially) active set $I' \neq I^e(x^0)$. In practice, it may be of interest to find $I'$. For example, in path-following methods, $I'$ could be used to compute the direction in which $P_c$ continues once the nonsmoothness in $x^0$ was detected. To this end, let $\{ g_1, \dots, g_k \}$ be the set of selection functions which are all essentially active at $x^0$. While it is not possible to determine $I'$ solely from the set $\conv(\{ \nabla f(x^0) \} \cup \partial g(x^0)) = \conv(\{ \nabla f(x^0) \} \cup \{ \nabla g_1(x^0), \dots, \nabla g_k(x^0) \})$, we can at least determine all potential candidates for $I'$ by finding all subsets $\{ i_1, \dots, i_m \} \subseteq \{ 1, \dots, k \}$ with
            \begin{align*}
                0 \in \conv(\{ \nabla f(x^0) \} \cup \conv(\{ \nabla g_{i_1}(x^0), \dots, \nabla g_{i_m}(x^0) \})).
            \end{align*}
        \end{remark}

\section{Examples} \label{sec:examples}
    
    In this section, we will show how our results from Section \ref{sec:structure_of_regularization_path} can be used to analyze the structure of regularization paths in two common applications. These are \emph{support vector machines} (SVMs) in data classification \cite{HTF2009} and the \emph{exact penalty method} in constrained optimization \cite{NW2006,PG1989}.
    
    \subsection{Support vector machine} \label{sec:SVM}
        
        Given a \emph{data set} $\{ (x^i, y^i) : x^i \in \R^l, y^i \in \{-1,1\}, i \in \{1,\dots,N\} \}$, the goal of the \emph{support vector machine} (SVM) is to find $w \in \R^l$ and $b \in \R$ such that
        \begin{align*}
            \sign(w^\top x^i + b) = y^i \quad \forall i \in \{1,\dots,N\}.
        \end{align*}
        In other words, the goal is to find a hyperplane $\{ x \in \R^l : w^\top x + b = 0 \}$ such that all $x^i$ with $y^i = 1$ lie on one side and all $x^i$ with $y^i = -1$ lie on the other side of the hyperplane. 
        Since such a hyperplane may not be unique, an additional goal is to find the one where the minimal distance of the $x^i$ to the hyperplane, also known as the \emph{margin}, is as large as possible. 
        One way of solving this problem is the penalization approach
        \begin{align} \label{eq:SVM}
            \min_{(w,b) \in \R^l \times \R} f(w,b) + \lambda g(w,b)
        \end{align}
        for $\lambda \geq 0$ and
        \begin{align*}
            &f : \R^l \times \R \rightarrow \R, \quad (w,b) \mapsto \frac{1}{2} \| w \|_2^2, \\
            &g : \R^l \times \R \rightarrow \R, \quad (w,b) \mapsto \sum_{i = 1}^N \max\{0, 1 - y^i (w^\top x^i + b) \}.
        \end{align*}
        
        Roughly speaking, minimizing $g$ ensures that the hyperplane separates the data, while minimizing $f$ maximizes the margin. In theory, the most favorable hyperplane would be the one with $g(w,b) = 0$ (if existent) and $f(w,b)$ as small as possible. But in practice, when working with large and noisy data sets, an imperfect separation where only few points violate the separation may be more desirable. The balance between the margin and the quality of the separation can be controlled via the parameter $\lambda$ in \eqref{eq:SVM}, yielding a regularization path $R_{\text{SVM}}$ as in \eqref{eq:regularization_path} (for $n = l+1$).
        
        \begin{remark}
            In the literature, the roles of $f$ and $g$ in problem \eqref{eq:SVM} are typically reversed. The resulting problem is equivalent to our formulation with the regularization parameter $\frac{1}{\lambda}$ (except for critical points of $f$ and $g$) (cf.\ Section 12.3.2 in \cite{HTF2009}). Nonetheless, when the regularization path of the SVM is considered, $\lambda$ in \eqref{eq:SVM} is more commonly used for its parametrization. 
        \end{remark}

        The structure of the regularization path of the SVM was already considered in earlier works. In \cite{HRT2004}, it was shown that $R_{\text{SVM}}$ is $1$-dimensional and piecewise linear up to certain degenerate points, and a path-following method was proposed that exploits this structure. It was conjectured (without proof) that the existence of these degenerate points is related to certain properties of the data points $(x^i,y^i)$, like having duplicates of the same point or having multiple points with the same margin. In \cite{OSY2010}, these degeneracies were analyzed further and a modified path-following method was proposed, specifically taking degenerate data sets into account. Other methods for degenerate data sets were proposed in \cite{DCM2013,SAG2016,WZC2019}. In the following, we will analyze how these degeneracies relate to the nonsmooth points we characterized in our results.

        Obviously, $f$ is twice continuously differentiable and $g$ is $PC^2$ with selection functions
        \begin{align*}
            \left\{ (w,b) \mapsto \sum_{i \in I} 1 - y^i (w^\top x^i + b) : I \subseteq \{1,\dots,N\} \right\}.
        \end{align*}
        Furthermore, both $f$ and $g$ are convex, so $R_{\text{SVM}}$ coincides with the critical regularization path (cf.\ \eqref{eq:critical_regularization_path}). Thus, we can apply our results from Section \ref{sec:structure_of_regularization_path} to analyze the structure of $R_{\text{SVM}}$. Since $f$ is quadratic and all selection functions are linear, Remark \ref{SM:rem_quadratic_piecewise_linear} shows that the regularization path is piecewise linear up to points violating the Assumptions \ref{assum:A1} to \ref{assum:A5}. Due to the properties of $g$, the Assumption \ref{assum:A1} always holds for the SVM, as shown in Remark \ref{SM:rem_SVM_A1} in the supplementary material.
        
        In the following, we will consider the remaining Assumptions \ref{assum:A2} to \ref{assum:A5} in the context of the SVM and relate them to the degeneracies reported in \cite{HRT2004}. We will do this by considering Example 1 from \cite{OSY2010}, which was specifically constructed to have a degenerate regularization path.
        
        \begin{example} \label{example:SVM_OSY2010_1}
            Consider the data set
            \begin{multline*}
                \left\{ ((0.7,0.3)^\top,1), ((0.5,0.5)^\top,1), ((2,2)^\top,-1),\right.\\ \left.((1,3)^\top,-1), ((0.75,0.75)^\top,1), ((1.75,1.75)^\top,-1) \right\}.
            \end{multline*}
            The regularization path for this data set can be computed analytically and is shown in Figure \ref{fig:reg_path_SVM_OSY2010_1}(a).
            \begin{figure}[ht] 
			    \centering 
			  	\subfloat[]{\includegraphics[width=0.49\textwidth]{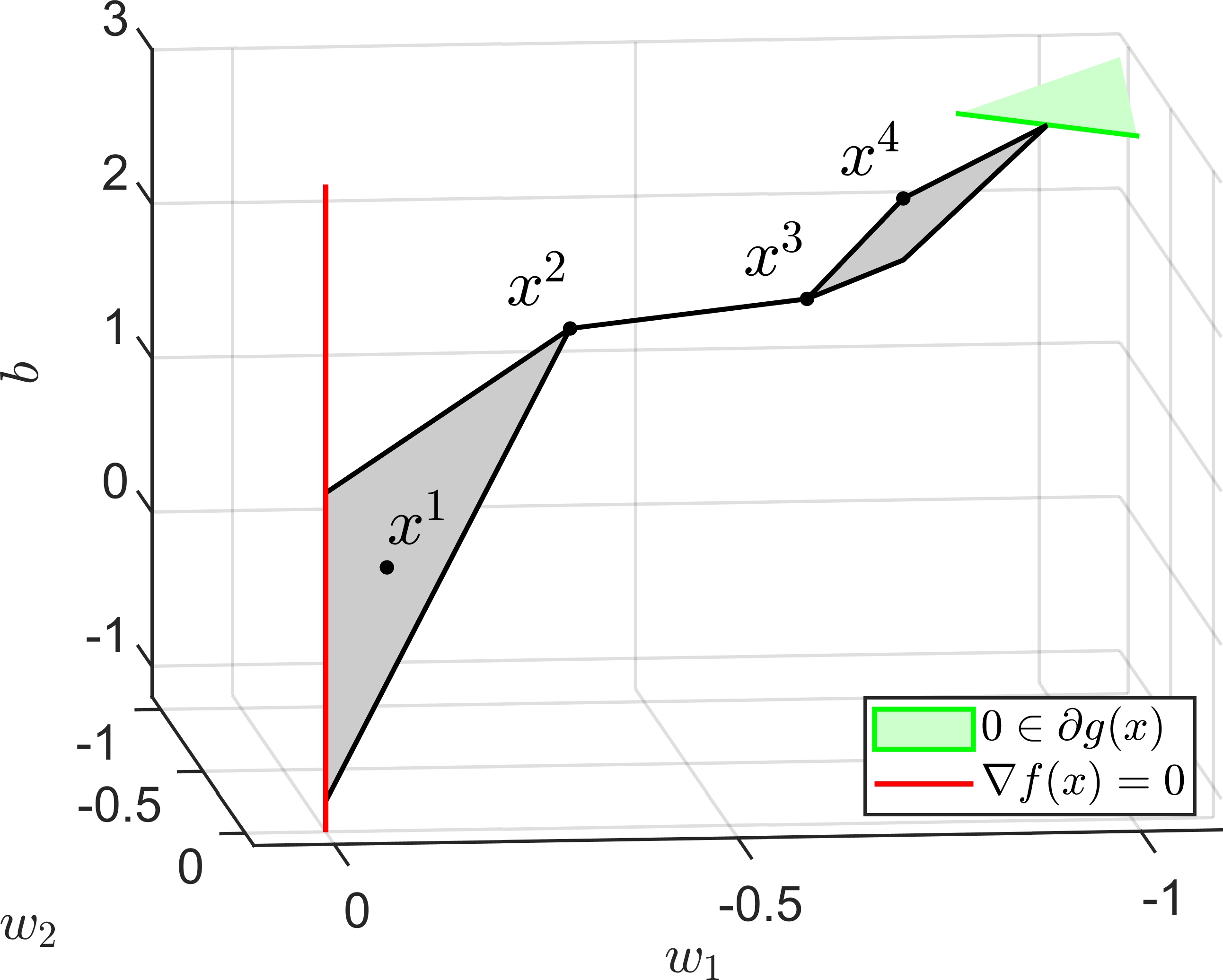}}
			  	\hfill
                \subfloat[]{\includegraphics[width=0.42\textwidth]{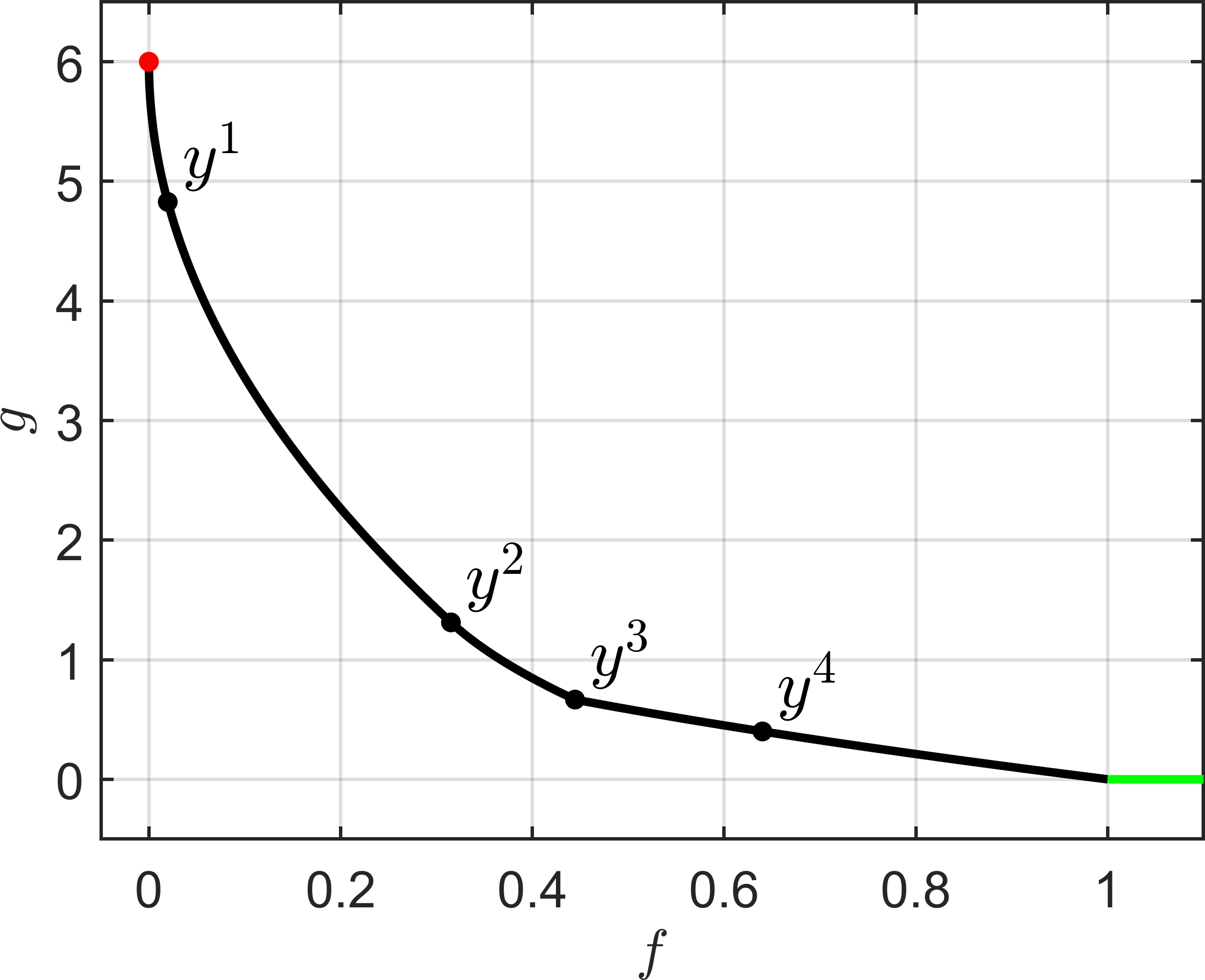}}
			    \caption{(a) Regularization path of the SVM in Example \ref{example:SVM_OSY2010_1} and the points $x^1 = \frac{1}{372} (-35,-65,137)^\top$, $x^2 = \frac{1}{93} (-35, -65, 137)^\top$, $x^3 = \frac{1}{3} (-2,-2,5)^\top$ and $x^4 = \frac{1}{5} (-4,-4,11)^\top$. (b) Image of the regularization path with $y^i = (f(x^i),g(x^i))^\top$, $i \in \{1,\dots,4\}$, and the same coloring as in (a).}
			    \label{fig:reg_path_SVM_OSY2010_1}
		    \end{figure}
		    In the following, we will analyze the points $x^1$, $x^2$, $x^3$ and $x^4$ highlighted in Figure \ref{fig:reg_path_SVM_OSY2010_1}(a) with respect to the Assumptions \ref{assum:A2} to \ref{assum:A5}.
		    
		    The point $x^1$ lies in one of the $2$-dimensional parts of the regularization path and it is possible to show that $g$ is smooth around $x^1$. It is easy to verify that Assumptions \ref{assum:A2}, \ref{assum:A3} and \ref{assum:A5} are satisfied. With regard to Assumption \ref{assum:A4}, it holds $r = \affdim(\aff(\{ \nabla f(x^1) \} \cup \partial g(x^1))) = 1$ (cf.\ Lemma~\ref{lem:Pc_caratheodory}) and
		    \begin{align*}
		        Dh(x,\alpha,\beta) = 
		        \begin{pmatrix}
		            2 \alpha     & 0          & 0 & -\frac{35}{372} & \frac{14}{5} \\[1pt]
		            0            & 2 \alpha   & 0 & -\frac{65}{372} & \frac{26}{5} \\
		            0            & 0          & 0 & 0               & 0 \\
		            0            & 0          & 0 & 1               & 1
		        \end{pmatrix}
		    \end{align*}
		    with $\rk(Dh(x,\alpha,\beta)) = 3$ for all $(x,\alpha,\beta) \in \R^n \times \R^{>0} \times \R^{>0}$. Thus, \ref{assum:A4}(b) holds which by Theorem \ref{thm:h_level_set_manifold} implies that the regularization path is the projection of an $n + r + 1 - m = 3 + 1 + 1 - 3 = 2$ dimensional manifold around $x^1$, as expected.
		    
		    The point $x^2$ lies in a kink in the regularization path. The subdifferential of $g$ in $x^2$ can be computed analytically and is shown in Figure \ref{fig:reg_path_SVM_OSY2010_2}(a).
		     \begin{figure}[ht] 
				\centering
				\subfloat[]{\includegraphics[width=0.33\textwidth]{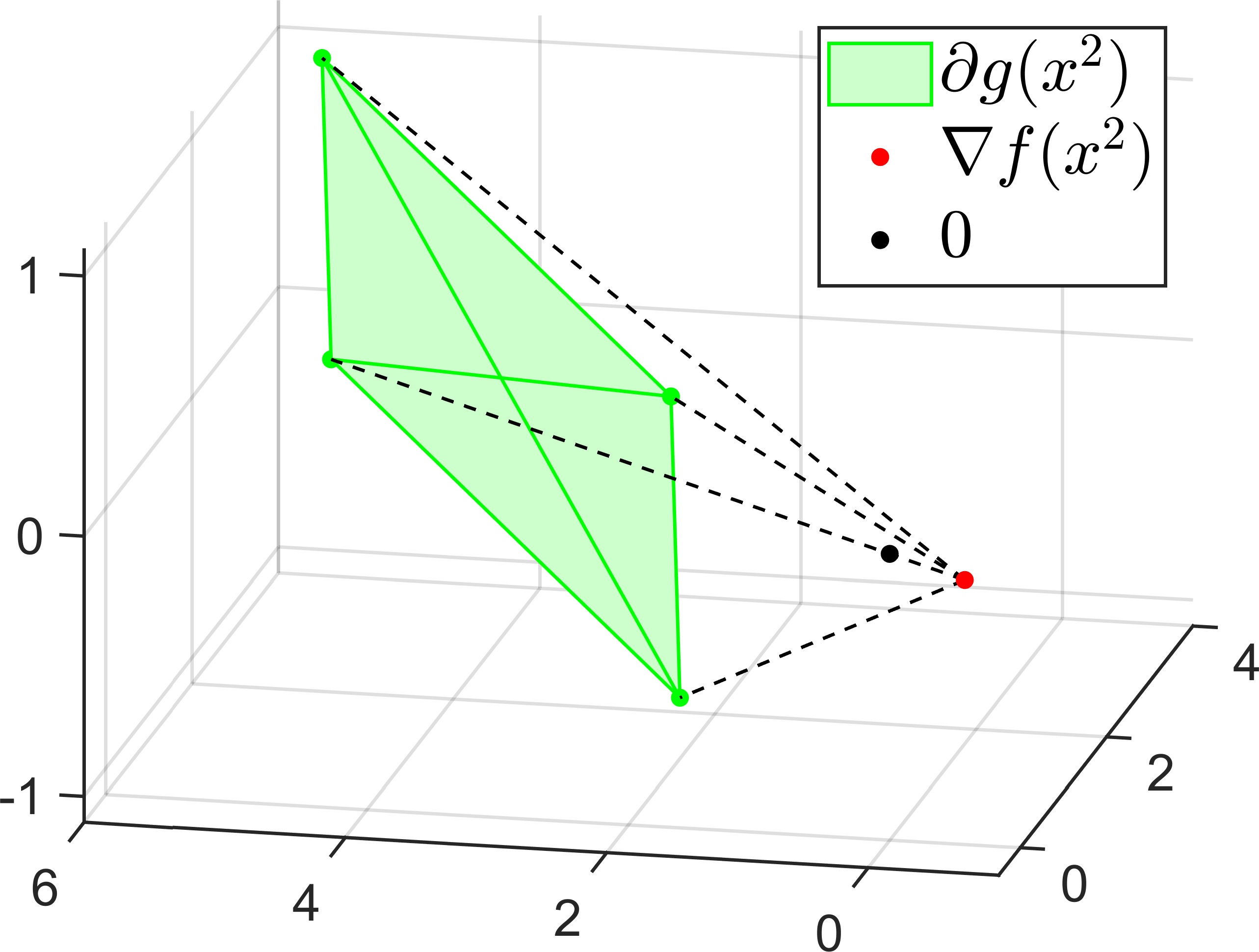}}
             	\subfloat[]{\includegraphics[width=0.33\textwidth]{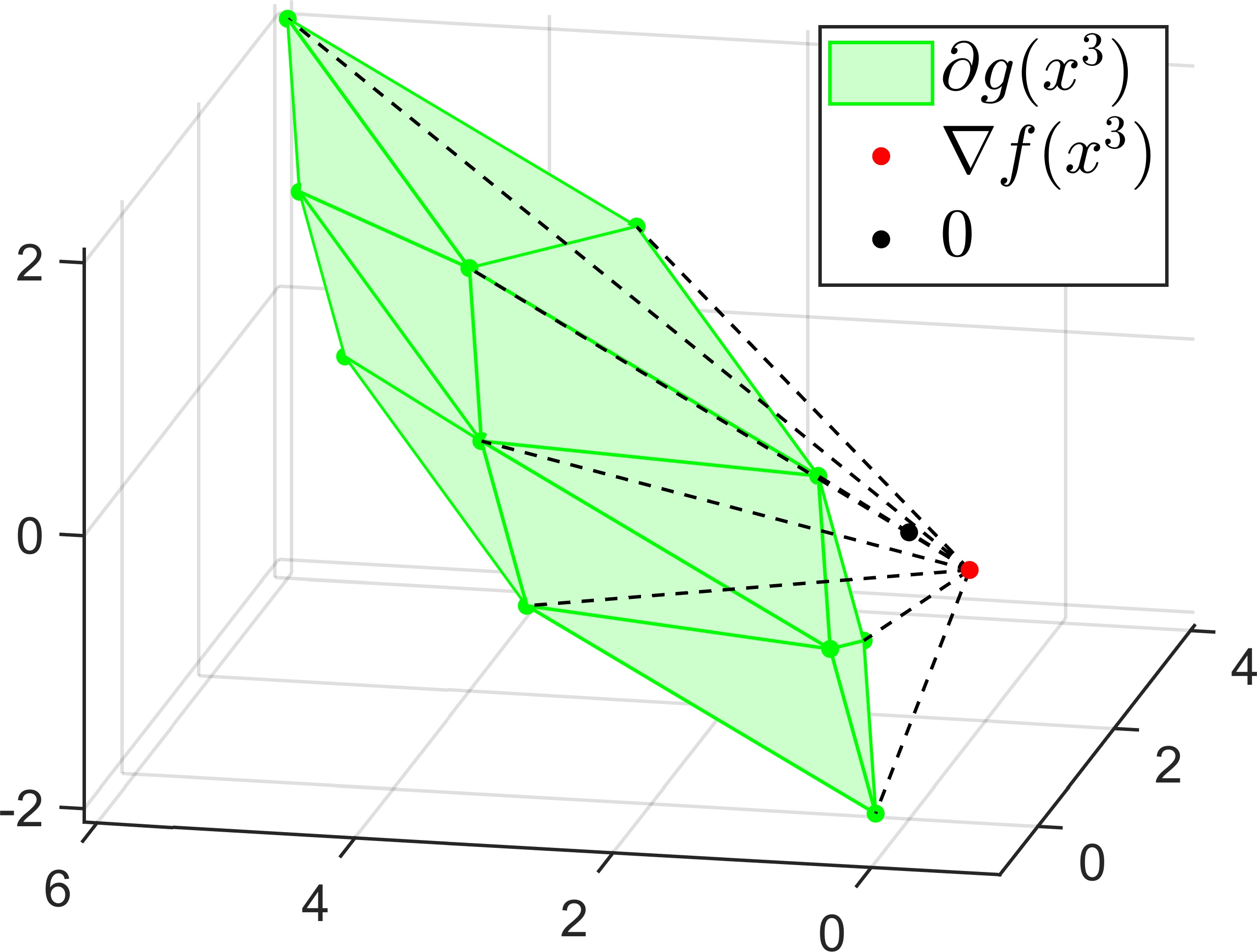}}
             	\subfloat[]{\includegraphics[width=0.33\textwidth]{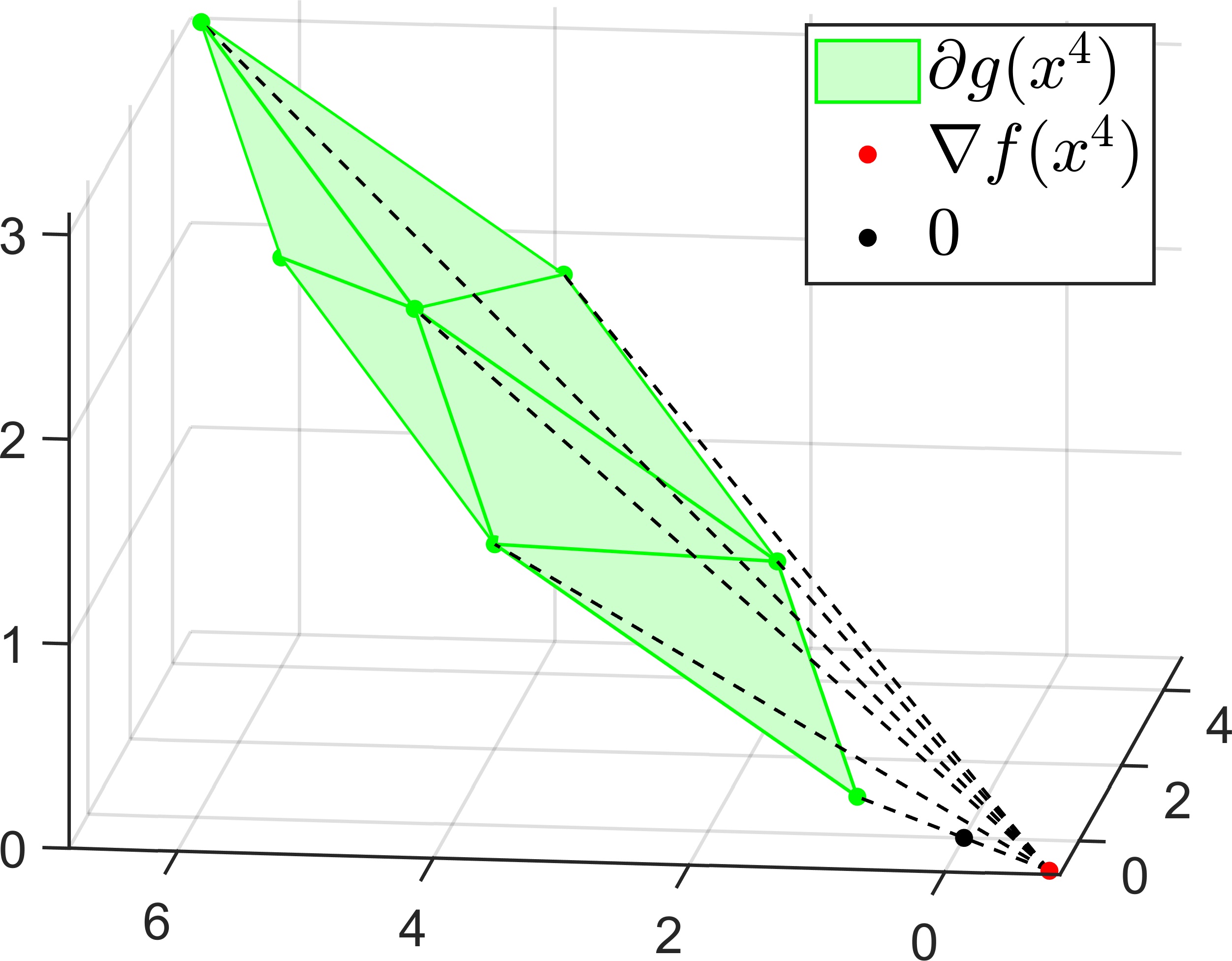}}
     			\caption{Gradient of $f$, subdifferential of $g$ and the (relative) boundary of the convex hull (dashed)  in $x^2$, $x^3$ and $x^4$ in Example \ref{example:SVM_OSY2010_1}.}
     			\label{fig:reg_path_SVM_OSY2010_2}
		    \end{figure}
		    In this case, we have $\affdim(\aff(\partial g(x^2))) = 2$ and $\nabla f(x^2) \notin \aff(\partial g(x^2))$, so Assumption \ref{assum:A2} holds. We see that zero lies on the relative boundary of $\conv(\{ \nabla f(x^2) \cup \partial g(x^2)\})$ such that Assumption \ref{assum:A3} must be violated (by Lemma \ref{lem:zero_in_relative_int}). Furthermore, it is possible to show that the active set changes in $x^2$, so Assumption \ref{assum:A5} is violated as well.
		   
		    The point $x^3$ lies in another kink of the regularization path. The corresponding subdifferential of $g$ is shown in Figure \ref{fig:reg_path_SVM_OSY2010_2}(b). As for $x^2$, Assumptions \ref{assum:A3} and \ref{assum:A5} are violated in $x^3$. But in contrast to $x^2$ we have $\affdim(\aff(\partial g(x^2))) = 3$, so $\nabla f(x^2) \in \aff(\partial g(x^2)) = \R^3$ trivially holds and Assumption \ref{assum:A2} is violated. As discussed in Remark \ref{rem:kink_front}, this results in a kink in the Pareto front in the image of  $x^3$ under the objective vector $(f,g)$, as can be seen in Figure \ref{fig:reg_path_SVM_OSY2010_1}(b).
		    
		    Finally, $x^4$ marks a corner of one of the $2$-dimensional parts of the regularization path and the corresponding subdifferential is shown in \ref{fig:reg_path_SVM_OSY2010_2}(c). As for $x^3$, Assumptions \ref{assum:A2}, \ref{assum:A3} and \ref{assum:A5} are violated in $x^4$. But unlike $x^3$, when we consider the image of $x^4$ in Figure \ref{fig:reg_path_SVM_OSY2010_1}(b), we see that the there is no kink in $y^4$. This suggests that the KKT multiplier of $f$ is unique even though Assumption \ref{assum:A2} is violated. Note that this is not a contradiction to Lemma \ref{lem:alpha1_unique} b), as $0$ lies on the relative boundary of $\conv(\{ \nabla f(x^4) \} \cup \partial g(x^4))$.
        \end{example}

    \subsection{Exact penalty method} \label{sec:exact_penalty_method}
        
        Consider the constrained optimization problem
        \begin{equation}\label{eq:constrained_OP}
            \begin{aligned}
                \min_{x \in \R^n} f(x) & \\
                s.t. \quad c^1_i(x) &\le 0, \quad i \in \{ 1,\dots,p \}, \\
                c^2_j(x) &= 0, \quad j \in \{ 1,\dots,q \},
            \end{aligned}
        \end{equation}
        where $f:\R^n \rightarrow \R$, $c^1_i:\R^n\rightarrow\R$, $i \in \{ 1,\dots,p \}$, and $c^2_j:\R^n \rightarrow \R$, $j \in \{ 1,\dots,q \}$, are continuously differentiable.
        In order to solve \eqref{eq:constrained_OP} the so-called \emph{exact penalty method} can be used, where the idea is to solve the (nonsmooth) problem
        \begin{equation}\label{eq:exact_penalty_function}
            \min_{x \in \R^n} f(x) + \lambda g(x)
        \end{equation}
        with a penalty weight $\lambda \geq 0$ and
        \begin{equation*}
            g : \R^n \rightarrow \R, \quad x \mapsto \left(\sum_{i=1}^p \max(c^1_i(x),0) + \sum_{j=1}^q |c^2_j(x)|\right).
        \end{equation*}
        It is easy to see that $g$ is $PC^1$ and that a set of selection functions is given by
        \begin{equation} \label{eq:exact_penalty_selection_functions}
            \left\{g_{\theta,\sigma}:\R^n\rightarrow \R, \ x \mapsto \sum_{i=1}^p \theta_i c_i^1(x) + \sum_{j =1}^{q} \sigma_i c_j^2(x) : \theta \in \{0,1\}^p, \sigma \in \{-1,1\}^q \right\}.
        \end{equation}
        
        The method is based on the theoretical result that there is some $\bar{\lambda}>0$ such that every strict local minimizer of \eqref{eq:constrained_OP} is a local minimizer of \eqref{eq:exact_penalty_function} for every $\lambda > \bar{\lambda}$, i.e., if $\lambda$ is large enough, then the constrained problem \eqref{eq:constrained_OP} can be solved via the unconstrained problem \eqref{eq:exact_penalty_function} (cf.\ \cite{NW2006}, Theorem 17.3). In practice, problem \eqref{eq:exact_penalty_function} will become ill-conditioned if $\lambda$ is large compared to $\bar{\lambda}$. Thus, it is instead solved for multiple, increasing values of $\lambda$ until a feasible solution is found.
        This results in a regularization path $R$ as in \eqref{eq:regularization_path}.
        Note that all feasible points of \eqref{eq:constrained_OP} are critical points of $g$ and the minimizer of \eqref{eq:constrained_OP} is typically the first intersection of the regularization path with the feasible set (when starting in the minimizer of $f$). In particular, the existence of $\bar{\lambda}$ as above implies that the minimizer of \eqref{eq:constrained_OP} is contained in $R$.
        
        In \cite{ZL2013}, $R$ is analyzed for the case where $f$ is quadratic (and strictly convex) and all $c^1_i$ and $c^2_j$ are affinely linear. In this case, $R$ coincides with the critical regularization path $R_c$ (cf.\ \eqref{eq:critical_regularization_path}). It is shown that $R$ is piecewise linear, which coincides with our results in Remark \ref{SM:rem_quadratic_piecewise_linear}. In \cite{ZL2015}, the more general case where $f$ and all $c_i^1$ are convex and all $c_j^2$ are affinely linear is considered. There, it still holds $R = R_c$ and it is shown that $R$ is piecewise smooth with kinks occurring where the constraints become satisfied or violated.
        
        Here, we want to use our theory to analyze the critical regularization path $R_c$ in the more general setting where $f$, $c_i^1$ and $c_j^2$ are merely continuously differentiable. By our results in Section \ref{sec:structure_of_regularization_path}, we know that $R_c$ is piecewise smooth up to points where the Assumptions \ref{assum:A1} to \ref{assum:A5} are violated. 
        In Remark \ref{SM:rem_exact_penalty_A1} in the supplementary material, it is shown that if all $x \in \R^n$ satisfy the \emph{linear independence constraint qualification} (LICQ), i.e., if 
        \begin{align} \label{eq:LICQ}
            \{ \nabla c_i^1(x) : c_i^1(x) = 0 \} \cup \{ \nabla c_j^2(x) : c_j^2(x) = 0 \}
        \end{align}
        is linearly independent for all $x \in \R^n$, then Assumption \ref{assum:A1} always holds and only Assumptions \ref{assum:A2} to \ref{assum:A5} may cause nonsmoothness in $R_c$.
        For these remaining assumptions we consider the following example, where the feasible set is given by continuously differentiable but nonconvex inequality constraints. It is inspired by problem (15) in \cite{ZL2015}.
        \begin{example} \label{ex:Exact_Penalty_Ex2}
            Consider the constrained optimization problem \eqref{eq:constrained_OP} with 
            \begin{equation*}
                \begin{aligned}
                    f(x) &= \frac{1}{2} x_1^2 + x_2^2 - x_1 x_2 + \frac{1}{2} x_1 - 2 x_2,\\
                    c^1_1(x) &= - \left( \left(x_1-\frac{1}{2} \right)^2 + x_2^2 - 1 \right),\\
                    c^1_2(x) &= \left( x_1+\frac{1}{2} \right)^2 + x_2^2 - 1,\\
                    c^1_3(x) &= - \left( x_1^2 + \left( x_2 - \frac{1}{2} \right)^2 - 1 \right).
                \end{aligned}
            \end{equation*}
            The corresponding critical regularization path $R_c$ of \eqref{eq:exact_penalty_function} can be computed analytically and is shown in black in Figure~\ref{fig:Exact_Penalty_Ex2}(a), consisting of two disconnected paths.
            \begin{figure}[ht] 
		    	\centering 
		    	\subfloat[]{\includegraphics[width=0.49\textwidth]{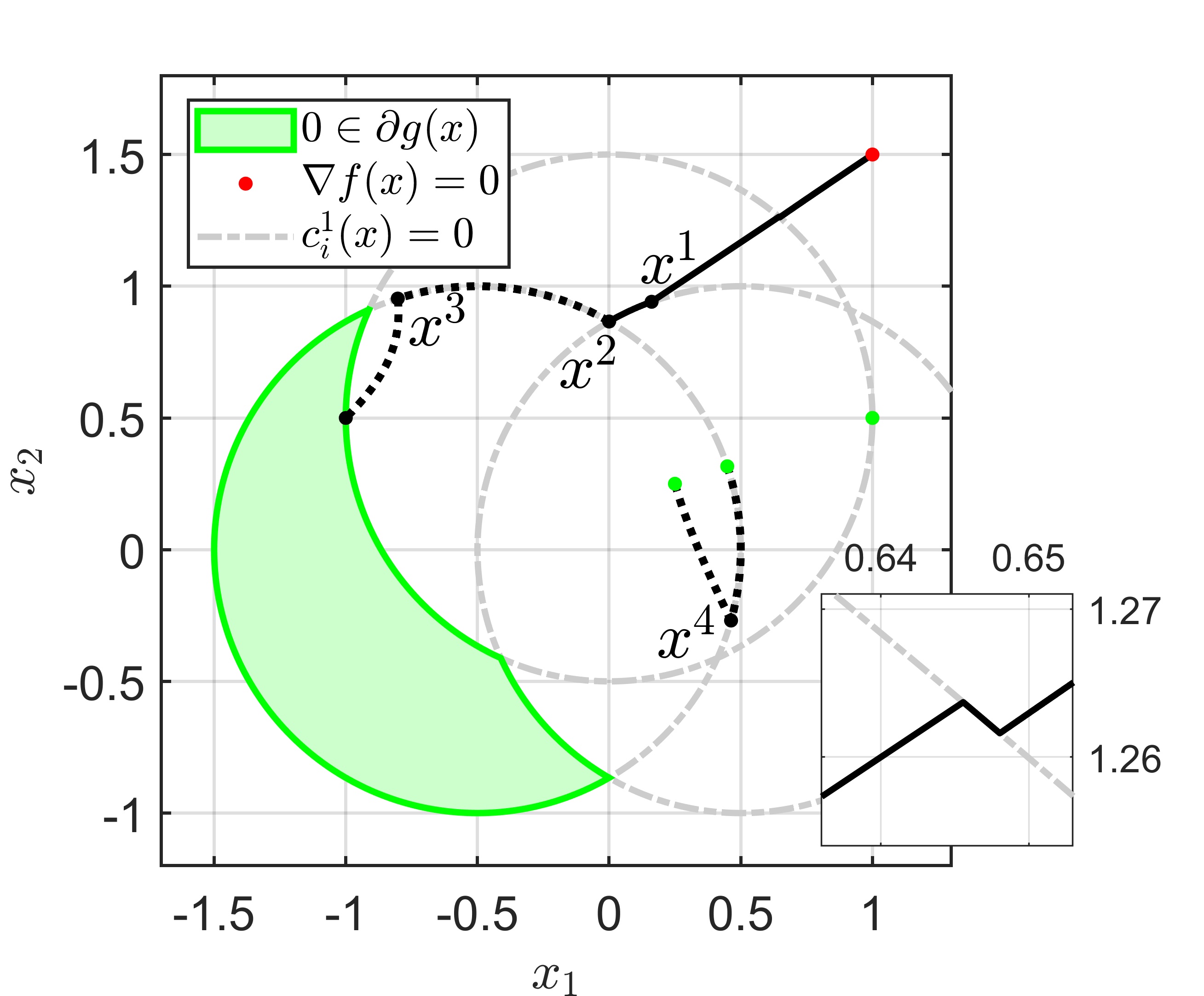}}
		    	\hfill
			    \subfloat[]{\includegraphics[width=0.49\textwidth]{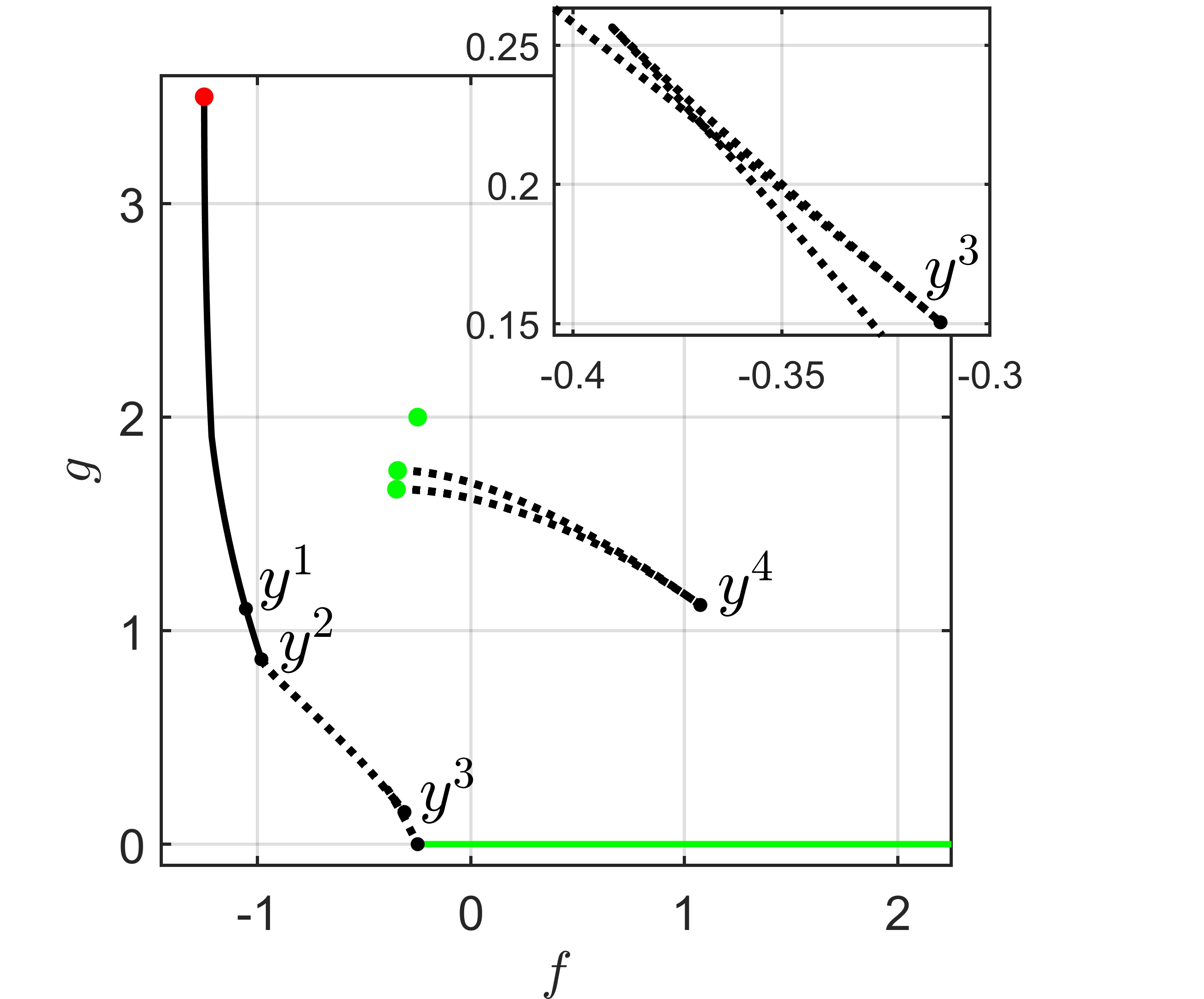}}
    		    \caption{(a) $R$ (black, solid) and $R_c$ (black) for the exact penalty method in Example~\ref{ex:Exact_Penalty_Ex2} and the points $x^1\approx(0.1614,0.9409)^\top$,
    		    $x^2=(0,\frac{\sqrt{3}}{2})^\top$,
    		    $x^3\approx(-0.8027, 0.9531)^\top$,
    		    $x^4\approx(0.4631, -0.2691)^\top$ with a zoom of the intersection of $c_3^1(x) = 0$ and $R_c$.
    		    (b) Image of $R_c$ with $y^i = (f(x^i),g(x^i))$, $i \in \{1,\dots,4\}$, and the same coloring as in (a). Furthermore, a zoom of the image around $y^3$.}
    		    \label{fig:Exact_Penalty_Ex2}
		    \end{figure}
            The feasible set of the constrained problem coincides with the critical set of $g$, excluding the three isolated critical points of $g$. Since $c_1^1$ and $c_3^1$ are nonconvex, $g$ is nonconvex as well, which is why $R_c$ does not coincide with the actual regularization path $R$ in this case. More precisely, $R$ is merely the union of the path from the minimal point of $f$ to $x^2$ and the intersection of $R_c$ with the feasible set (cf.\ Figure \ref{fig:Exact_Penalty_Ex2}).
            
            In the following we will analyze the kinks of $R_c$, which are located in $x^1$ to $x^4$ and between the minimal point of $f$ and $x^1$ (cf.\ Figure~\ref{fig:Exact_Penalty_Ex2}(a)). First of all, it is easy to see that kinks occur precisely where constraints become satisfied or violated along $R_c$. Due to the construction of the selection functions (cf.\ \eqref{eq:exact_penalty_selection_functions}), this causes Assumption \ref{assum:A5} to be violated in these points. 
            
            For $x^1$, the gradient of $f$ and the subdifferential of $g$ are shown in Figure \ref{fig:Exact_Penalty_Ex2_subdiff}(a). We see that Assumption \ref{assum:A2} holds and that Assumption \ref{assum:A3} is violated since zero lies on the relative boundary of $\conv(\{ \nabla f(x^1) \} \cup \partial g(x^1))$ (cf.\ Lemma \ref{lem:zero_in_relative_int}). The same behavior occurs in all other kinks except $x^2$. For $x^2$, $\nabla f(x^2)$ and $\partial g(x^2)$ are shown in Figure \ref{fig:Exact_Penalty_Ex2_subdiff}(b). In contrast to the other points, Assumption \ref{assum:A2} is clearly violated since $\dim(\aff(\partial g(x^2))) = 2 = n$. As discussed in Remark \ref{rem:kink_front}, this causes a kink in the image of $R_c$, which can be seen in Figure \ref{fig:Exact_Penalty_Ex2}(b). Moreover, zero lies in the relative interior of $\conv(\{ \nabla f(x^2) \} \cup \partial g(x^2))$ and it is easy to see that Assumption \ref{assum:A3} holds. 
            
            In addition to the features described so far, the image of $R_c$ possesses so-called \emph{turning points}. If we treat the image of $R_c$ as an actual (continuous) path, then these are points where the direction of the path abruptly turns around. For example, this can be observed in $y^3$ and $y^4$ in Figure \ref{fig:Exact_Penalty_Ex2}(b). These points were already discussed in \cite{BGP2021} and in Example 3.4 therein, it was highlighted that they are not necessarily caused by any nonsmoothess of the objectives. Since we are mainly interested in the structure of $R_c$ in this article, we will leave their analysis for future work.
            
            \begin{figure}[ht] 
		    	\centering 
		    	\subfloat[]{\includegraphics[width=0.49\textwidth]{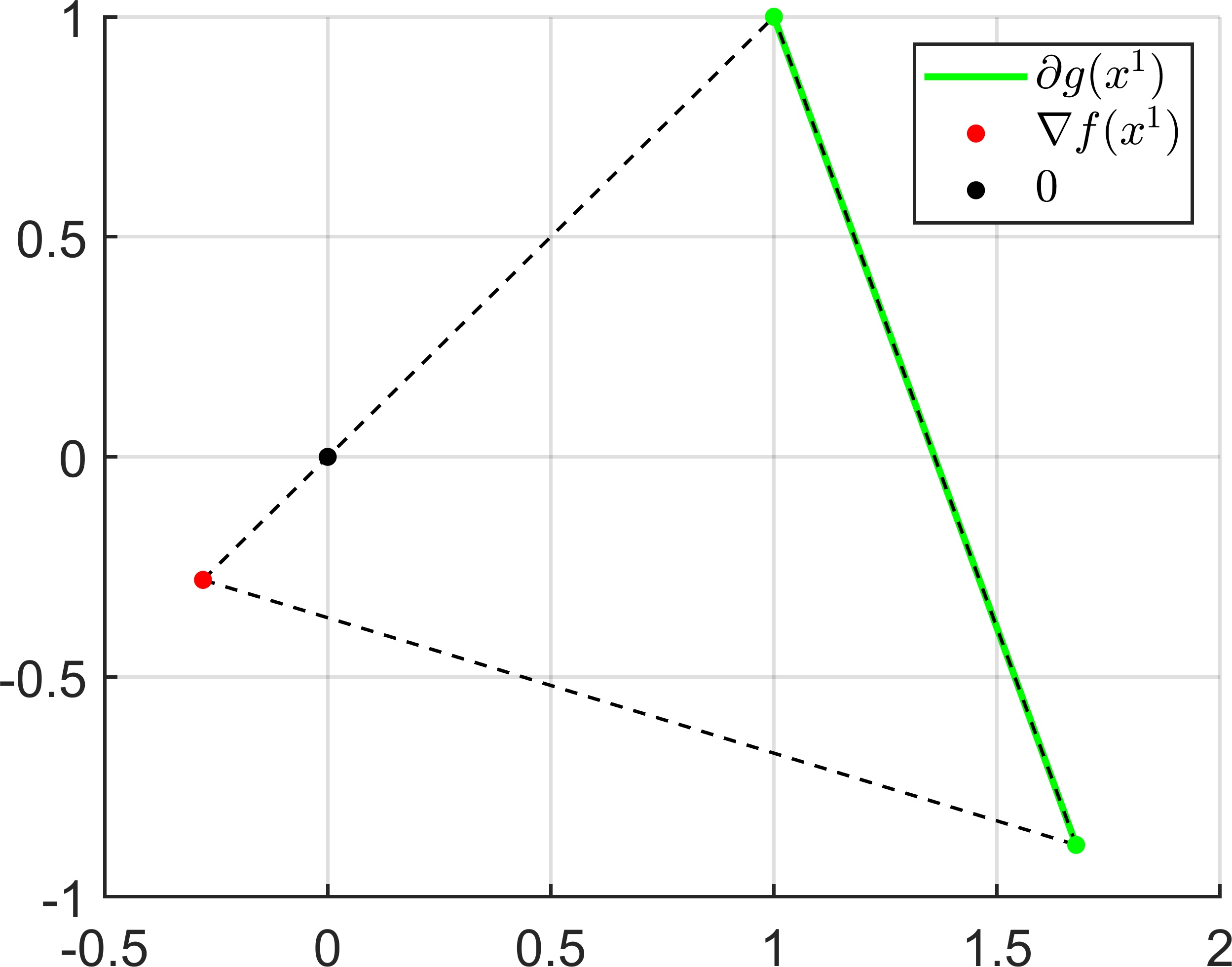}}
		    	\hfill
			    \subfloat[]{\includegraphics[width=0.49\textwidth]{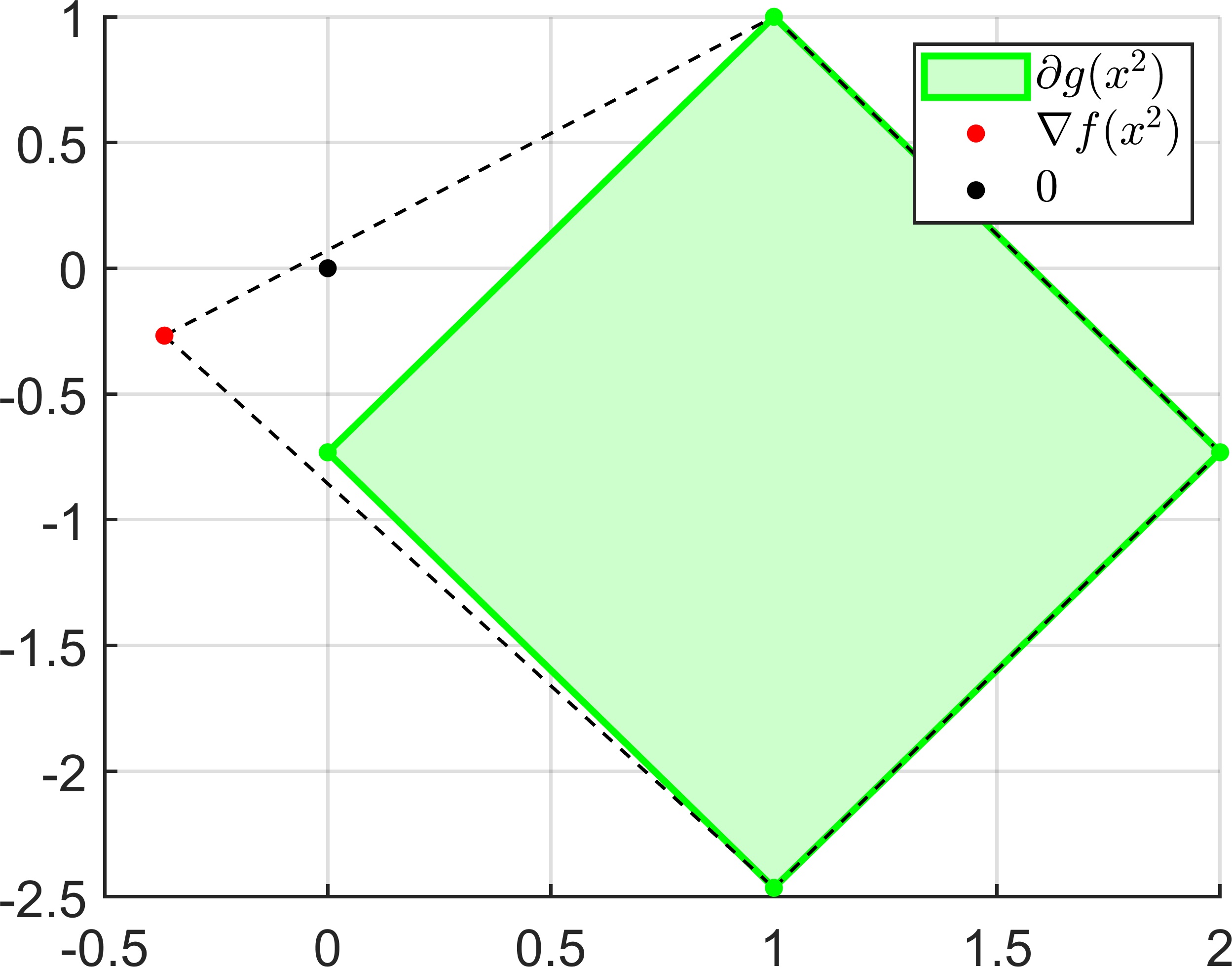}}
    		    \caption{Gradient of $f$, subdifferential of $g$ and the corresponding (relative) boundary of the convex hull (dashed) in $x^1$ and $x^2$ of Example~\ref{ex:Exact_Penalty_Ex2}.}
    		    \label{fig:Exact_Penalty_Ex2_subdiff}
		    \end{figure}
        \end{example}
        
        Note that all kinks in the previous examples were points where constraints become satisfied or violated, which suggests that the structural results from \cite{ZL2015} also hold in our more general nonconvex case, at least for the critical regularization path $R_c$. Furthermore, $R_c$ is still connecting the minimum of $f$ to the solution of the constrained problem \eqref{eq:constrained_OP} (which is the intersection of $R_c$ with the feasible set). Thus, it might be possible to apply a path-following method similar to the one in \cite{ZL2015} to nonconvex problems as well.

\section{Conclusion} \label{sec:conclusion}

    In this article, we have presented results about the structure of regularization paths for piecewise differentiable regularization terms. We did this by first showing that the critical regularization path is related to the Pareto critical set $P_c$ of the multiobjective optimization problem which contains the objective function $f$ and the regularization term $g$. Afterwards, we analyzed $P_c$ by reformulating it as a union of the intersection of certain sets, which allowed us to apply differential geometry to obtain structural results. During this derivation, we identified five assumptions (\ref{assum:A1} to \ref{assum:A5}) which (when combined) are sufficient for $P_c$ to have a smooth structure locally around a given $x^0 \in P_c$. In turn, nonsmooth features of $P_c$ (like ``kinks'') can be classified depending on which of these five assumptions is violated. We demonstrated this by analyzing the regularization paths for the support-vector machine and the exact penalty method.

    Based on our results in this article, there are multiple possible directions for future work:
    \begin{itemize}
        \item We believe that most of our theoretical results would still hold (with only minor adjustments) if we would assume $f$ to be merely piecewise differentiable as well. (In this case, the regularization function $f + \lambda g$ would still be piecewise differentiable.)
        \item Although the MOP \eqref{eq:MOP} we considered in this article has only two objectives, multiobjective optimization can handle any number of objectives. In particular, \eqref{eq:MOP} could be formulated for arbitrarily many regularization terms. We believe that results similar to ours (with a higher-dimensional regularization path) could be obtained for this case. This would allow regularization methods such as the \emph{elastic net} \cite{ZH2005} to be incorporated into our framework.
        \item While we were focused on regularization in this article, our results can also be used in the context of general multiobjective optimization to construct path-following methods for the solution of nonsmooth MOPs, extending \cite{BGP2021, ZL2015, RZ2007, H2001}.
        \item Although we provided the main ingredients for the construction of path-following methods, i.e., a way to compute the tangent space in smooth areas and a characterization of nonsmooth points, their development and actual implementation is still non-trivial. For example, other important ingredients are the computation of new points on $R$ after taking a step along the tangent direction (also known as a \emph{corrector}) and the computation of the correct tangent direction after a kink was found. Treating these problems in our general framework could greatly simplify the development of new path-following methods. 
    \end{itemize}

\clearpage
\begin{center}
	\bfseries\MakeUppercase{Supplementary Materials: On the structure of regularization paths for piecewise differentiable regularization terms}
\end{center}

\setcounter{section}{0}
\renewcommand{\thesection}{SM\arabic{section}}
\section{Proof of Lemma \ref{lem:alpha1_unique}} \label{SM:proof_lem_affine_coeff_unique}
    Let $\{g_1,\dots,g_k\}$ be a set of selection functions of $g$ and let $I^e(x^0) = \{ i_1, \dots, i_l \}$. \\
    a) By assumption, for $s \in \{1,2\}$, there have to be $\alpha^s_1 > 0$ and $\beta^s \in (\R^{\geq 0})^l$ such that $\alpha^s_1 + \sum_{j = 1}^l \beta^s_j = 1$,
    \begin{align} \label{eq:lem_alpha1_unique_1}
        \alpha^s_1 \nabla f(x^0) + \sum_{j = 1}^l \beta^s_j \nabla g_{i_j}(x^0) = 0
    \end{align}
    and $\alpha_1^1 \neq \alpha_1^2$. This implies
    \begin{align*}
        &\alpha^1_1 \nabla f (x^0) + \sum_{j = 1}^l \beta^1_j \nabla g_{i_j}(x^0) = \alpha^2_1 \nabla f(x^0) + \sum_{j = 1}^l \beta^2_j \nabla g_{i_j}(x^0)\\
        \Leftrightarrow \quad &\nabla f(x^0) = \frac{1}{\alpha_1^1 - \alpha_1^2} \sum_{j=1}^l (\beta^2_j-\beta^1_j)\nabla g_{i_j}(x^0) = \sum_{j=1}^l \frac{\beta^2_j-\beta^1_j}{\alpha_1^1 - \alpha_1^2} \nabla g_{i_j}(x^0)
    \end{align*}
    with
    \begin{align*}                
        \sum_{j=1}^l\frac{\beta^2_j-\beta^1_j}{\alpha_1^1 - \alpha_1^2} = \frac{1 - \alpha_1^2 - (1 - \alpha_1^1)}{\alpha_1^1 - \alpha_1^2} = 1,
    \end{align*}
    showing that $\nabla f(x^0) \in \aff(\partial g(x^0))$. \\
    b) Since $\nabla f(x^0) \in \aff(\partial g(x^0))$ there has to be some $\beta' \in \R^l$ with $\sum_{j = 1}^l \beta'_j = 1$ and
    \begin{align} \label{eq:lem_alpha1_unique_2}
        \nabla f(x^0) = \sum_{j = 1}^l \beta'_j \nabla g_{i_j}(x^0).
    \end{align}
    Furthermore, by Lemma \ref{lem:relint_polytope}, zero being contained in $\ri(\conv(\{ \nabla f(x^0) \} \cup \partial g(x^0)))$ is equivalent to the existence of $\alpha_1 > 0$ and $\beta \in (\R^{>0})^l$ with $\alpha_1 + \sum_{j = 1}^l \beta_j = 1$ and
    \begin{align} \label{eq:lem_alpha1_unique_3}
        \alpha_1 \nabla f(x^0) + \sum_{j = 1}^l \beta_j \nabla g_{i_j}(x^0) = 0.
    \end{align}
    Combination of \eqref{eq:lem_alpha1_unique_2} and \eqref{eq:lem_alpha1_unique_3} yields
    \begin{align*}
        (\alpha_1 - \lambda) \nabla f(x^0) + \sum_{j = 1}^l (\beta_j + \lambda \beta'_j) \nabla g_{i_j}(x^0) = 0 \quad \forall \lambda \in \R
    \end{align*}
    and
    \begin{align*}
        (\alpha_1 - \lambda) + \sum_{j = 1}^l (\beta_j + \lambda \beta'_j) 
        = \alpha_1 + \sum_{j = 1}^l \beta_j + \lambda \left( - 1 +  \sum_{j = 1}^l \beta'_j \right)
            = 1 \quad \forall \lambda \in \R.
    \end{align*}
    Since $\alpha_1 > 0$ and $\beta \in (\R^{>0})^l$, there has to be some $\lambda \neq 0$ such that $\alpha_1 - \lambda > 0$ and $\beta + \lambda \beta' \in (\R^{>0})^l$. In particular, $\alpha_1 - \lambda \neq \alpha_1$ is another KKT multiplier corresponding to $f$ in $x^0$, completing the proof.    

\section{Proof of Lemma \ref{lem:zero_in_relative_int}} \label{SM:proof_lem_zero_in_relative_int}
    Let $\{g_1, \dots, g_k\}$ be a set of selection functions that satisfies \ref{assum:A3}. Let $L := \{1,\dots,k\} \setminus \{i_1,\dots,i_r\}$. Since $\{ \nabla f(x^0) \} \cup \{ \nabla g_i(x^0) : i \in \{i_1, \dots, i_r \} \}$ is an affine basis of $\aff(\{ \nabla f(x^0) \} \cup \partial g(x^0))$, there are coefficients $\theta^l \in \R$ and $\nu^l \in \R^r$ for every $l \in L$ with $\theta^l + \sum_{j = 1}^r \nu^l_j = 1$ and 
    \begin{align*}
        \nabla g_l(x^0) = \theta^l \nabla f(x^0) + \sum_{j = 1}^r \nu^l_j \nabla g_{i_j}(x^0).
    \end{align*}
    Let $\bar{\theta} := -\sum_{l \in L} \theta^l$ and $\bar{\nu}_j := -\sum_{l \in L} \nu^l_j$ for $j \in \{1,\dots,r\}$. Then
    \begin{align*}
        0 &= \sum_{l \in L} \left( \nabla g_l(x^0) - \theta^l \nabla f(x^0) - \sum_{j = 1}^r \nu^l_j \nabla g_{i_j}(x^0) \right) \\
        &= \bar{\theta} \nabla f(x^0) + \sum_{j = 1}^r \bar{\nu}_j \nabla g_{i_j}(x^0) + \sum_{l \in L} \nabla g_l(x^0)
    \end{align*}
    and $\bar{\theta} + \sum_{j = 1}^r \bar{\nu}_j + \sum_{l \in L} 1 = 0$. Let $\alpha^0 > 0$ and $\beta^0 \in (\R^{>0})^r$ as in \ref{assum:A3}. Then
    \begin{align} \label{eq:lem_zero_in_relative_int_1}
        0 &= \alpha^0 \nabla f(x^0) + \sum_{j = 1}^r \beta^0_j \nabla g_{i_j}(x^0) \nonumber \\
        &= \alpha^0 \nabla f(x^0) + \sum_{j = 1}^r \beta^0_j \nabla g_{i_j}(x^0) + \lambda \left( \bar{\theta} \nabla f(x^0) + \sum_{j = 1}^r \bar{\nu}_j \nabla g_{i_j}(x^0) + \sum_{l \in L} \nabla g_l(x^0) \right) \nonumber \\
        &= (\alpha^0 + \lambda \bar{\theta}) \nabla f(x^0) + \sum_{j = 1}^r (\beta_j^0 + \lambda \bar{\nu}_j) \nabla g_{i_j}(x^0) + \sum_{l \in L} \lambda \nabla g_l(x^0)
    \end{align}
    for all $\lambda \in \R$. By construction, there is some $\lambda > 0$ such that \eqref{eq:lem_zero_in_relative_int_1} is a vanishing convex combination with strictly positive coefficients. Applying Lemma \ref{lem:relint_polytope} completes the proof.

\section{Remark regarding Section \ref{sec:structure_of_PcI_OmegaI}} \label{SM:rem_quadratic_piecewise_linear}
    Let
    \begin{align*}
        f : \R^n \rightarrow \R, \quad x \mapsto \frac{1}{2} x^\top A x + b^\top x + c
    \end{align*}
    for $A \in \R^{n \times n}$, $b \in \R^n$ and $c \in \R$. Furthermore, assume that there is a set of selection functions $\{ g_1, \dots, g_k \}$ of $g$ consisting of affinely linear functions, i.e.,
    \begin{align*}
        g_i : \R^n \rightarrow \R, \quad x \mapsto d_i^\top x + e_i
    \end{align*}
    for $d_i \in \R^n$, $e_i \in \R$, $i \in \{1,\dots,k\}$. Let $x^0 \in P_c$ and assume that Lemma \ref{lem:intersec_level_set} is applicable, yielding an index set $\{ i_1, \dots, i_r \} \subseteq \{1,\dots,k\}$, an open neighborhood $U' \subseteq \R^n$ of $x^0$ and coefficients $\alpha^0 \in \R^{>0}$ and $\beta^0 \in (\R^{>0})^r$ such that $h(x^0,\alpha^0,\beta^0) = 0$. In this case, the map $h$ reduces to
    \begin{align*}
        h(x,\alpha,\beta) = 
        \begin{pmatrix}
            \alpha (Ax + b) + \sum_{j = 1}^r \beta_j d_{i_j} \\
            \alpha + \sum_{j = 1}^r \beta_j - 1 \\
            ((d_{i_j}^\top - d_{i_1})^\top x + e_{i_j} - e_{i_1})_{j \in \{2,\dots,r\}}
        \end{pmatrix}.
    \end{align*}
    We will show that
    \begin{align} \label{eq:rem_quadratic_piecewise_linear_1}
        {\pr}_x(h^{-1}(0)) \cap U' = ( x^0 + {\pr}_x(\ker(Dh(x^0,\alpha^0,\beta^0))) ) \cap U',
    \end{align}
    which by Lemma \ref{lem:intersec_level_set} implies that $P_c^{I^e(x^0)} \cap \Omega^{I^e(x^0)} \cap U'$ is an affinely linear set with dimension $\dim({\pr}_x(\ker(Dh(x^0,\alpha^0,\beta^0))))$.
        
    To this end, let $(v^x,v^\alpha,v^\beta) \in \ker(Dh(x^0,\alpha^0,\beta^0))$. Since $\alpha^0 > 0$, there is some $\varepsilon > 0$ such that $\alpha^0 - t v^\alpha > 0$ for all $t \in [0,\varepsilon)$. Define
    \begin{align*}
        s : [0,\varepsilon) \rightarrow \R, \quad t \mapsto \frac{t \alpha^0}{\alpha^0 - t v^\alpha}.
    \end{align*}
    Since $\alpha^0 > 0$ and $\beta^0 \in (\R^{>0})^r$, there is some $\varepsilon' \in (0,\varepsilon)$ such that   
    \begin{align*}
        &\alpha^0 + s(t) v^\alpha > 0, \\
        &\beta_j^0 + s(t) v_j^\beta > 0 \quad \forall j \in \{1,\dots,r\}
    \end{align*}
    for all $t \in [0,\varepsilon')$. Furthermore, since $U'$ is an open neighborhood of $x^0$, there is some $\varepsilon'' > 0$ such that $x^0 + t v^x \in U'$ for all $t \in [0,\varepsilon'')$. Finally, a simple calculation shows that
    \begin{align*}
        h(x^0 + t v^x, \alpha^0 + s(t) v^\alpha, \beta^0 + s(t) v^\beta) = 0 \quad \forall t \in [0,\varepsilon'').
    \end{align*}
    Thus, ``$\supseteq$'' holds in \eqref{eq:rem_quadratic_piecewise_linear_1}. \\
    In turn, let $(x^1,\alpha^1,\beta^1) \in U' \times \R^{>0} \times (\R^{>0})^r$ with $h(x^1,\alpha^1,\beta^1) = 0$. Let $s := \frac{\alpha^0}{\alpha^1}$. It is easy to show that
    \begin{align*}
        (x^1 - x^0, s (\alpha^1 - \alpha^0), s (\beta^1 - \beta^0)) \in \ker(Dh(x^0, \alpha^0, \beta^0)),
    \end{align*}
    implying that ``$\subseteq$'' holds in \eqref{eq:rem_quadratic_piecewise_linear_1}. 

\section{Proof of Lemma \ref{lem:PcIOmegaI_touch}} \label{SM:proof_lem_PcIOmegaI_touch}
    By assumption there is a sequence $(x^s)_s \in P_c$ with $\lim_{s \rightarrow \infty} x^s = x^0$ and $I^e(x^s) \neq I^e(x^0)$ for all $s \in \N$. Assume w.l.o.g.\ that $I^e(x^s)$ is constant for all $s \in \N$. 
    Due to the definition of the essentially active set, we can assume w.l.o.g.\ that $I^e(x^s) = \{ i_1, \dots, i_m \} \subseteq I^e(x^0)$ for some $m < l$. Since $x^s \in P_c$ for all $s \in \N$, there are sequences $(\alpha^s)_s \in \R^{\geq 0}$, $(\beta^s)_s \in (\R^{\geq 0})^m$ with $\alpha^s + \sum_{j = 1}^m \beta^s_j = 1$ and
    \begin{align*}
        \alpha^s \nabla f(x^s) + \sum_{j = 1}^m \beta^s_j \nabla g_{i_j}(x^s) = 0
    \end{align*}
    for all $s \in \N$. Since $(\alpha^s)_s$ and $(\beta^s)_s$ are bounded, we can assume w.l.o.g.\ that there are $\alpha \in \R^{\geq 0}$ and $\bar{\beta} \in (\R^{\geq 0})^m$ with $\lim_{s \rightarrow \infty} \alpha^s = \alpha$ and $\lim_{s \rightarrow \infty} \beta^s = \bar{\beta}$. In particular, $\alpha + \sum_{j = 1}^l \bar{\beta}_j = 1$. By continuity of $\nabla f$ and $\nabla g_i$, $i \in \{1,\dots,k\}$, we have
    \begin{align*}
        \alpha \nabla f(x^0) + \sum_{j = 1}^m \bar{\beta}_j \nabla g_{i_j}(x^0) = 0.
    \end{align*}
    The proof follows by setting $\beta = (\bar{\beta}_{i_1},\dots,\bar{\beta}_{i_m},0,\dots,0)^\top \in (\R^{\geq 0})^l$.

\section{Remark regarding Section \ref{sec:SVM}} \label{SM:rem_SVM_A1}
    Let $g : \R^n \rightarrow \R$ be any piecewise linear and convex function. Let $x^0 \in \R^n$. By Lemma \ref{lem:local_ess_active}, there is an open neighborhood $U \subseteq \R^n$ of $x^0$ and a set of (affinely linear) selection functions $\{ g_1, \dots, g_k \}$ of $g|_U$ which are all essentially active in $x^0$. In particular, $I(x^0) = \{1,\dots,k\}$, so \ref{assum:A1}(i) holds. \\
    To see that \ref{assum:A1}(ii) holds, let $z \in U$ and $j \in I(z)$. Since all selection functions are essentially active in $x^0$, we have
    \begin{align*}
        x^0 \in \cl(\interior(\{ y \in U : g(y) = g_j(y)\})),
    \end{align*}
    so $V := \interior(\{ y \in U : g(y) = g_j(y)\}) \neq \emptyset$. Let $y \in V$. Since $g$ is convex and $g_j$ is affinely linear, we have
    \begin{equation}  \label{eq:SVM_A1_1}
        \begin{aligned}
            g((1-\lambda)y + \lambda z) &\leq (1-\lambda) g(y) + \lambda g(z) = (1-\lambda) g_j(y) + \lambda g_j(z) \\
            &= g_j((1-\lambda)y + \lambda z) \quad \forall \lambda \in [0,1].
        \end{aligned}
    \end{equation}
    Assume that we have inequality in \eqref{eq:SVM_A1_1}, i.e., assume that there is some $\bar{\lambda} \in [0,1]$ with $g(\bar{x}) < g_j(\bar{x})$ for $\bar{x} := (1-\bar{\lambda})y + \bar{\lambda} z$.
    Then
    \begin{align*}
        g((1-\lambda)y + \lambda \bar{x}) &\leq (1-\lambda) g(y) + \lambda g(\bar{x}) < (1-\lambda) g(y) + \lambda g_j(\bar{x}) \\
        &= g_j((1-\lambda)y + \lambda \bar{x}) \quad \forall \lambda \in (0,1].
    \end{align*}
    This is a contradiction to the openness of $V$,
    so we must have equality in \eqref{eq:SVM_A1_1}. This implies
    \begin{align*}
        j \in I((1-\lambda)y + \lambda z) \quad \forall \lambda \in [0,1].
    \end{align*}
    As this holds for arbitrary $y \in V$, we have
    \begin{align*}
        j \in I(x) \quad \forall x \in \conv(V \cup \{ z \}).
    \end{align*}
    Since $V$ is open in $\R^n$, it is possible to show that
    \begin{align*}
        z \in \cl(\interior(\conv(V \cup \{ z \}))) \subseteq \cl(\interior(\{ y \in U : g(y) = g_j(y)\})),
    \end{align*}
    showing that $j \in I^e(z)$. \\
    Since $\nabla g_i$ is constant for all $i \in \{1,\dots,k\}$, it is easy to see that \ref{assum:A1}(iii) holds as well.

\section{Remark regarding Section \ref{sec:exact_penalty_method}} \label{SM:rem_exact_penalty_A1}
    We begin by deriving an explicit expression for the active set. To this end, let $x \in \R^n$ and assume w.l.o.g.\ that there are $\bar{p} \in \{1,\dots,p\}$, $\bar{q} \in \{1,\dots,q\}$ such that 
    \begin{equation} \label{eq:rem_exact_penalty_A1_2}
        \begin{aligned}
            c^1_i(x) = 0, \quad \forall i \in \{1,\dots,\bar{p}\}, \qquad 
            &c^1_i(x) \neq 0, \quad \forall i \in \{\bar{p}+1,\dots,p\},\\
            c^2_j(x) = 0, \quad \forall j \in \{1,\dots,\bar{q}\}, \qquad
            &c^2_j(x) \neq 0, \quad \forall j \in \{\bar{q}+1,\dots,q\}.
        \end{aligned}
    \end{equation}
    For $i \in \{ \bar{p}+1, \dots, p \}$ and $j \in \{ \bar{q}+1, \dots, q \}$ define
    \begin{equation} \label{eq:rem_exact_penalty_A1_1}
        \begin{aligned}
            \hat{\theta}_i &:= 
            \begin{cases}
                1, & \text{if } c_i^1(x) > 0 \\
                0, & \text{if } c_i^1(x) < 0
            \end{cases}, \\
            \hat{\sigma}_j &:= \sign(c^2_j(x)),
        \end{aligned}
     \end{equation}
    and
    \begin{align*}  
        \bar{c} : \R^n \rightarrow \R, \quad x \mapsto \sum_{i=\bar{p}+1}^p \hat{\theta}_i c_i^1(x) + \sum_{j = \bar{q}+1}^{q} \hat{\sigma}_j c_j^2(x).
    \end{align*}
    Then by construction,
    \begin{align*}
        g(x) &= \sum_{i=1}^p \max \{ c^1_i(x),0 \} + \sum_{j=1}^q |c^2_j(x)| = \sum_{i=\bar{p}+1}^p \max \{ c^1_i(x),0 \} + \sum_{j=\bar{q}+1}^q |c^2_j(x)| \\
        &= \sum_{i=\bar{p}+1}^p \hat{\theta}_i c_i^1(x) + \sum_{j = \bar{q}+1}^{q} \hat{\sigma}_j c_j^2(x) = \bar{c}(x) \\
        &= \bar{c}(x) + \sum_{i=1}^{\bar{p}} \bar{\theta}_i c_i^1(x) + \sum_{j=1}^{\bar{q}} \bar{\sigma}_j c_j^2(x) \\
        &= g_{(\bar{\theta},\hat{\theta}),(\bar{\sigma},\hat{\sigma})}(x)
    \end{align*}
    for all $\bar{\theta} \in \{0,1\}^p$, $\bar{\sigma} \in \{-1,1\}^q$. Thus
    \begin{align*}
        \bar{I} := \left\{
        ((\bar{\theta},\hat{\theta})^\top,(\bar{\sigma},\hat{\sigma})^\top) :
        \bar{\theta} \in \{0,1\}^{\bar{p}}, \bar{\sigma} \in \{-1,1\}^{\bar{q}} \right\} \subseteq I(x).
    \end{align*}
    To show that ``$\supseteq$'' holds, let $(\tilde{\theta},\tilde{\sigma}) \in I(x)$. Then
    \begin{align*}
        \hat{\theta}_i - \tilde{\theta}_i = 
        \begin{cases}
            -1, & \text{if } \hat{\theta}_i \neq \tilde{\theta}_i, \ \hat{\theta}_i = 0 \\ 
            1, & \text{if } \hat{\theta}_i \neq \tilde{\theta}_i, \ \hat{\theta}_i = 1 \\
            0, & \text{otherwise}
        \end{cases},
        \quad
        \hat{\sigma}_j - \tilde{\sigma}_j = 
        \begin{cases}
            -2, & \text{if } \hat{\sigma}_j \neq \tilde{\sigma}_j, \ \hat{\sigma}_i = -1 \\ 
            2, & \text{if } \hat{\sigma}_j \neq \tilde{\sigma}_j, \ \hat{\sigma}_i = 1 \\
            0, & \text{otherwise}
        \end{cases}
    \end{align*}
    for all $i \in \{\bar{p}+1, \dots, p\}$, $j \in \{\bar{q}+1, \dots, q\}$. Combined with \eqref{eq:rem_exact_penalty_A1_1}, this implies
    \begin{align*}
        0 &= g(x) - g_{\tilde{\theta},\tilde{\sigma}}(x) = \sum_{i=1}^p \max \{ c^1_i(x),0 \} + \sum_{j=1}^q |c^2_j(x)| - \sum_{i=1}^p \tilde{\theta}_i c_i^1(x) - \sum_{j =1}^{q} \tilde{\sigma}_j c_j^2(x) \\ 
        &= \sum_{i=\bar{p}+1}^p (\hat{\theta}_i - \tilde{\theta}_i) c_i^1(x) + \sum_{j = \bar{q}+1}^{q} (\hat{\sigma}_j - \tilde{\sigma}_j) c_j^2(x) \\
        &= \sum_{\substack{i=\bar{p}+1\\ \hat{\theta}_i \neq \tilde{\theta}_i}}^p |c_i^1(x)| + 
        \sum_{\substack{j = \bar{q}+1\\ \hat{\sigma}_j \neq \tilde{\sigma}_j}}^{q} 2 |c_j^2(x)|,
    \end{align*}
    so both sums must be empty, i.e., $\hat{\theta}_i = \tilde{\theta}_i$ for all $i \in \{ \bar{p}+1, \dots, p \}$ and $\hat{\sigma}_j = \tilde{\sigma}_j$ for all $i \in \{ \bar{q}+1, \dots, q \}$. In particular $(\tilde{\theta},\tilde{\sigma}) \in \bar{I}$, so $\bar{I} = I(x)$ for all $x \in \R^n$.
    
    In the following, we will show that all active selection functions are essentially active. To this end, let $(\theta,\sigma) \in I(x) = \bar{I}$. Define
    \begin{align*}
        v^i &:= 
        \begin{cases}
            \nabla c_i^1(x), & \text{if } \theta_i = 0 \\
            -\nabla c_i^1(x), & \text{if } \theta_i = 1
        \end{cases} \quad \forall i \in \{1,\dots,\bar{p}\},  \\
        w^j &:= -\sigma_j \nabla c_j^2(x) \quad \forall j \in \{1,\dots,\bar{q}\}, \\
        C &:= \conv( \{ v^i : i \in \{1,\dots,\bar{p}\} \} \cup \{ w^j : j \in \{1,\dots,\bar{q}\} \} ).
    \end{align*}
    The LICQ (cf. \eqref{eq:LICQ}) implies that $0 \notin C$. With a basic result from convex analysis (cf. Lemma in \cite{CG1959}), it follows that there is some $d \in \R^n \setminus \{ 0 \}$ with
    \begin{align*}
        0 &> \langle v^i, d \rangle =
        \begin{cases}
            \langle \nabla c_i^1(x), d \rangle, & \text{if } \theta_i = 0 \\
            - \langle \nabla c_i^1(x), d \rangle, & \text{if } \theta_i = 1
        \end{cases} \quad \forall i \in \{1,\dots,\bar{p}\}, \\
        0 &> \langle w^j, d \rangle = -\langle \sigma_j \nabla c_j^2(x), d \rangle \quad \forall j \in \{1,\dots,\bar{q}\}.
    \end{align*}
    The continuity of the constraint functions implies that there is some $T > 0$ such that
    \begin{align*}
        \sign(c_i^1(x + t d)) &= 
        \begin{cases}
            -1, & \text{if } \theta_i = 0 \\
            1, & \text{if } \theta_i = 1
        \end{cases} \quad \forall i \in \{1,\dots,p\}, \\
        \sign(c_j^2(x + t d)) &= \sigma_j \quad \forall j \in \{1,\dots,q\},
    \end{align*}
    for all $t \in (0,T)$. 
    Note that in particular, for all points $x + t d$ with $t \in (0,T)$, there is a neighborhood of $x + t d$ on which $g$ is smooth with $g = g_{\theta,\sigma}$. This shows that $(\theta,\sigma) \in I^e(x)$. 
    
    Let $x^0 \in \R^n$. From our discussion up to this point it follows that \ref{assum:A1}(i) and (ii) hold for an appropriate open neighborhood $U$ of $x^0$. To show that \ref{assum:A1}(iii) holds, let $(\theta',\sigma')$ be any element of $\bar{I} = I(x^0)$ (with $\bar{p}$ and $\bar{q}$ as in \eqref{eq:rem_exact_penalty_A1_2}) and $z \in U$. Clearly,
    \begin{multline}
        \label{eq:rem_exact_penalty_A1_3}
            \spn( \{ \nabla g_{\theta,\sigma}(z) - \nabla g_{\theta',\sigma'}(z) : (\theta,\sigma) \in \bar{I} \} ) \\
            \subseteq \spn( \{ \nabla c_i^1(z) : i \in \{ 1, \dots, \bar{p} \} \} \cup \{ \nabla c_j^2(z) : j \in \{ 1, \dots, \bar{q} \} \} ).
    \end{multline}
    We will show that we actually have equality in \eqref{eq:rem_exact_penalty_A1_3}, which implies that \ref{assum:A1}(iii) holds by the LICQ (cf. \eqref{eq:LICQ}).
    To this end, let $i' \in \{1,\dots,\bar{p} \}$. Define
    \begin{align*}
        \tilde{\theta}_i := 
        \begin{cases}
            \theta'_i, & \text{if } i \neq i' \\
            1, & \text{if } i = i', \theta'_i = 0 \\
            0, & \text{if } i = i', \theta'_i = 1
        \end{cases}
        \quad \forall i \in \{1,\dots,\bar{p}\}.
    \end{align*}
    Then $(\tilde{\theta},\sigma') \in \bar{I}$, so $g_{\tilde{\theta},\sigma'}(z) - g_{\theta',\sigma'}(z) = \pm \nabla c_{i'}^1(z)$ and $\nabla c_{i'}^1(z)$ is contained in the left-hand side of \eqref{eq:rem_exact_penalty_A1_3}. Analogously, it is possible to show that $\nabla c_j^2(z)$ is contained in the left-hand side of \eqref{eq:rem_exact_penalty_A1_3} for all $j \in \{1,\dots,\bar{q}\}$, such that equality holds.

\bibliographystyle{siamplain}
\bibliography{references}
\end{document}